\RequirePackage{fix-cm}
\documentclass[smallextended,numbook,runningheads]{svjour3}       % onecolumn (second format)
\smartqed  % flush right qed marks, e.g. at end of proof
\usepackage{graphicx}%%%%%%%%%%%%%% ÃÃ­ÅÃÃÅÃÃ± %%%%%%%%%%%%%%%%%%%%%%%5
\usepackage{amssymb}
\usepackage{amsmath}
\usepackage[ruled]{algorithm2e}
\usepackage{color}
\usepackage{multirow}
\usepackage{enumerate}
\usepackage{geometry}
\usepackage{tikz}
\usetikzlibrary{positioning, shapes.geometric}
\usepackage{subfigure}

\newtheorem{Theorem}{Theorem}[section]
\newtheorem{Assumption}{Assumption}[section]
\newtheorem{Lemma}[Theorem]{Lemma}
\newtheorem{Corollary}[Theorem]{Corollary}

\newtheorem{Remark}[Theorem]{Remark}

\newtheorem{Example}[Theorem]{Example}

\newtheorem{Definition}[Theorem]{Definition}

\newcommand{\di}{d}

\newcommand{\R}{\mathbb{R}}
\newcommand{\E}[1]{\mathbb{E}\left[#1\right]}

\newcommand{\br}[1]{\left(#1\right)}
\newcommand{\abs}[1]{\left\vert#1\right\vert}
\newcommand{\norm}[1]{\left\|#1\right\|}
\newcommand{\bbr}[1]{\left\{#1\right\}}

\allowdisplaybreaks
\begin{document}

\title{An efficient gradient projection method for stochastic optimal control problem with expected integral state constraint\thanks{The work of the authors was supported by the National Natural Science Foundation of China (12171042).}
}

\titlerunning{Gradient projection method for SOCP with integral state constraint}% if too long for running head

\author{Qiming Wang \and Wenbin Liu}

%\authorrunning{Q. Wang\and A. Fang\and T. Zhou\and W. Liu} % if too long for running head

\institute{Qiming Wang \at
School of Mathematical Sciences, Beijing Normal University, Zhuhai 519087, People's Republic of China. \\
\email{202231130040@mail.bnu.edu.cn}           %  \\
% \emph{Present address:} of F. Author  %  if needed
\and
Wenbin Liu \at
Research Center for Mathematics, Beijing Normal University, Zhuhai 519087, People's Republic of China.
\and
Faculty of Business and Management, Beijing Normal University-Hong Kong Baptist University United International College, Zhuhai 519087, People's Republic of China.\\
\email{wbliu@uic.edu.cn}
}

\date{Received: date / Accepted: date}

\maketitle

\begin{abstract}
In this work, we present an efficient gradient projection method for solving a class of stochastic optimal control problem with expected integral state constraint. 
The first order optimality condition system consisting of forward-backward stochastic differential equations and a variational equation is first derived.
Then, an efficient gradient projection method with linear drift coefficient is proposed where the state constraint is guaranteed by constructing specific multiplier.
Further, the Euler method is used to discretize the forward-backward stochastic differential equations and the associated conditional expectations are approximated by the least square Monte Carlo method, yielding the fully discrete iterative scheme. Error estimates of control and multiplier are presented, showing that the method admits first order convergence. Finally we present numerical examples to support the theoretical findings.

\end{abstract}

\section{Introduction}
In this work, we consider the numerical method for stochastic optimal control problems (SOCPs) with expected integral state constraint. 
To be specific, let $T \in (0, \infty)$ be a fixed time horizon and $(\Omega, \mathcal{F}, \{\mathcal{F}_t\}_{0 \leq t \leq T}, \mbox{P})$ be a complete probability space with the natural filtration $\{\mathcal{F}_t\}_{0\leq t \leq T}$ generated by a $d$-dimensional Brownian motion $W_t$. 
The considered SOCP is given by
\begin{equation}\label{object}
    \min\limits_{(y_t^u,u)\in K\times L^2([0,T];\mathbb{R}^{m})} J(u) := \E{\int_{0}^{T} \left(h(y_t^u) + j\br{u}\right) \di t + k(y_T^u)}.
\end{equation}
Here $u:[0, T] \to \R^m$ is the control variable and $J:L^2([0,T]; \R^m)\to\R$ is the continuously Fr\'{e}chet differentiable objective functional. $h,k: \R^n \to \R$, $j: \R^m \to \R$ are smooth functions and $y_t^u: [0, T] \times \Omega \to \R^n$ is the state process solved by the following stochastic differential equation (SDE):
\begin{equation}\label{state}
\left\{\begin{aligned}
  d y_t^u& = b(y_t^u,u)dt+\sigma(y_t^u,u) dW_t,\quad t\in (0,T],\\
  y_0^u &= y^0\in \mathbb{R}^{n},
\end{aligned}\right.
\end{equation}
where $b: \mathbb{R}^{n} \times \mathbb{R}^{m} \to \mathbb{R}^{n}$ and $\sigma: \mathbb{R}^{n}\times\mathbb{R}^{m}\to\mathbb{R}^{n\times d}$ are the drift and diffusion coefficients, respectively. The state constraint set is given by  
\begin{eqnarray}
\begin{aligned}\label{state_con}
    K := \left\{y_t \in L_{\mathcal{F}}^2\left([0,T] \times \Omega;\mathbb{R}^n\right) \Big\vert\; \int_0^T \E{y_t} \di t \leq \delta\in\mathbb{R}^{n}\right\}.
\end{aligned}
\end{eqnarray}
Here $L^2([0,T]; \R^m)$ denotes the space consisting of all functions $u: [0, T] \to \R^m$ that satisfy
$\norm{u}^2 := \int_0^T \abs{u(t)}^2 \di t < +\infty$ and $L_{\mathcal{F}}^2([0,T] \times \Omega; \R^n)$ denotes the space consisting of all $\{\mathcal{F}_t\}_{0\leq t\leq T}$ adapted processes $y_t: [0, T] \times \Omega \to \R^n$ that satisfy
$\norm{y_t}_{\Omega}^2 := \int_0^T\E{ \abs{y_t}^2 }\di t < +\infty$.

SOCPs have extensive applications in numerous areas, including financial mathematics \cite{jinrong1,jinrong2,jinrong3}, engineering systems \cite{gongcehng1,gongcehng2}, biological and medical applications \cite{shengwu}. Generally, explicit solutions for most SOCPs are not feasible, especially for those with constraints. Therefore, effective numerical methods play a crucial role in practical applications. Two general frameworks for solving SOCPs are dynamic programming principle (DPP) \cite{dpp1,dpp2} and stochastic maximum principle (SMP) \cite{max1,max2,max3,max4}. 
In this work, SMP is employed to solve problem (\ref{object})-(\ref{state}) due to the fact that state constraint may cause non-continuity of the value function, leading to the non-applicability of the DPP. The principle of SMP is to convert the directional derivative of the objective functional into a more efficiently computed form called variational equation by introducing an adjoint equation, which is a backward stochastic differential equation (BSDE). Then the variational equation coupled with
the forward-backward stochastic differential equations (FBSDEs), i.e., state and adjoint equations, forms an optimality condition system that can be used to solve the optimal control problem. 

The key to solving the optimality condition system is to solve BSDEs effectively. 
One of the most popular ways is to solve BSDEs backwardly through time, which leads to the development of numerous spatial-temporal discretization schemes for BSDEs. 
From a perspective of temporal discretization, we can use the Euler-type methods \cite{euler1}, generalized $\theta$-schemes \cite{theta1,theta2,theta3}, Runge-Kutta schemes \cite{runge1}, multistep schemes \cite{mu1,mu2,mu3} and so on. From a perspective of spatial discretization, approximation techniques can be roughly classified into grid-type method and regression approach. For the former, we can refer to \cite{runge1,mu1,grid1,grid2,grid3,max3} and their references. A common class of methods used in regression approach is the least square Monte Carlo (LSMC) method, which is introduced in \cite{intro} in the context of American options and is applied to the backward scheme in \cite{goblt1}. 
Then a forward scheme for simulating BSDEs is introduced in \cite{bender1}, where the conditional expectations are approximated by LSMC and it avoids high order nestings of conditional expectations backwards in time. \cite{goblt2,goblt3} investigate the numerical solution of BSDEs with data dependent on a jump-diffusion process by LSMC. It's worth noting that the explicit bounds for the time step, the number of simulated paths and the number of functions in the error estimate are derived, which can enable us to optimize and adjust parameters to achieve the desired accuracy. For more studies on BSDEs and LSMC method, we refer to \cite{other1,other2,other3,other4} and their reference.

The theoretical analysis of SOCPs with state constraint has been studied since the foundation of stochastic optimal control theory. The special case where the diffusion term of the state equation does not contain control variable was first studied by \cite{state1} and \cite{state2}. The stochastic maximum principle of SOCP with terminal state constraints is presented in \cite{state3}.
In \cite{state4}, a stochastic optimal control problem is studied where the controlled system is described by FBSDEs, while the forward state is constrained in a convex set at the terminal time. The maximum principle is deduced and the fully coupled case is studied in \cite{state5}. 
Theoretical analysis of state-constrained SOCPs has been developing \cite{state6,state7}, however, the work on numerical analysis of the SOCPs with state constraint remains unreported. 
We make the first attempt to design numerical algorithm for solving SOCP problem with state constraint.

In this work, we first derive the optimality condition system by constructing the Lagrange functional. In the case where drift coefficient is a linear function, the gradient projection method is presented. Compared with SOCP with control constraint, state constraint is an implicit constraint with respect to control variable, which makes the gradient projection method more complex. Then, for the FBSDEs in the optimality condition system, we discretize them by Euler method in time and the conditional expectations are approximated by LSMC method to obtain the fully discrete iterative scheme. Further, the first order convergence of the control and multiplier is deduced under certain parameter settings. Finally, we verify the theoretical analysis by several numerical experiments. 

The paper is organized as follows. In the next section, some notation and assumptions are given and optimality condition system is deduced. In section 3, the gradient projection method is proposed and the temporal discretization for the control is performed. In section 4, the fully discrete iterative scheme and algorithm are given based on the fully discrete scheme of FBSDEs. The error estimates for the control and multiplier are derived. Finally, several numerical examples are performed to verify the theoretical findings in section 5.

\section{Optimality condition system}\label{section_opt}
In this section, we derive the first-order optimality condition. 
For the sake of notational simplicity, our discussion will be confined to the one-dimensional case, i.e., $m = n = d = 1$, however, the entire framework can be trivially extended to the multi-dimensional case. 

We begin with the following notation:
\begin{enumerate}[$\bullet$]
    \item $C_{b}^{j,j,j,l}$: the set of continuously differentiable functions $(y,p,q,t) \in \mathbb{R}\times\mathbb{R}\times\mathbb{R}\times[0,T] \mapsto g(y, p, q, t) \in \mathbb{R}$ with bounded partial derivative functions $\partial_{y}^{j_1}\partial_{p}^{j_2}\partial_{q}^{j_3}\partial_{t}^{l_1} g$ for $0 \leq j_1 + j_2 + j_3 \leq j$ and $0\leq l_1\leq l$. Analogous definitions apply for $C_{b}^{j,l}$ and $C_{b}^{j}$.
    \item $C_b^{j+\alpha}$: the set consisting of all $g \in C_{b}^{j}$ with $g^{(j)}$ being H\"{o}lder continuous with index $\alpha \in (0, 1)$. 
    \item $\mathrm{int}(A)$: the interior of the set $A \subset \R$.  
    \item $x + A := \bbr{x + a \vert\; a \in A}$ for $x \in \R$ and $A \subset \R$. 
    \item $(u, v)$: the inner product of $u, v \in L^2([0,T]; \mathbb{R})$, i.e., $(u,v) := \int_{0}^{T} u(t)v(t) \di t$.
    \item $[y, w]$: the inner product of $y, w \in L_{\mathcal{F}}^2([0,T]\times\Omega;\mathbb{R})$, i.e., $[y, w] := \int_{0}^{T} \mathbb{E}\left[y_t w_t\right] dt$.
%    \item $\mathbb{E}_{t}^{y}[\,\cdot\,] := \mathbb{E}\left[\,\cdot\,|\mathcal{F}_{t},y_{t}=y\right]$ for $t \in [0, T]$ and $y \in \R$. 
\end{enumerate} 

For the functions in SOCP (\ref{object})-(\ref{state}), the following assumptions are given:
\begin{Assumption}\label{jiashe1}
The functions $b, \sigma, h, j$ and $k$ in the SOCP \eqref{object}-\eqref{state} satisfy
\begin{enumerate} [$\bullet$]
\item $b=b(y,u)$ and $\sigma=\sigma(y,u)$ are continuously differentiable with
respect to $y$ and $u$ and there exist positive constants $C_{b}$ and $C_{\sigma}$ such that
$|b_y^{\prime}|+|b_u^{\prime}|\leq C_{b}$ and $|\sigma_y^{\prime}|+|\sigma_u^{\prime}|\leq C_{\sigma}.$
\item $h, j$ and $k$ are continuously differentiable, and their derivatives
have at most a linear growth with respect to the underlying variables.
\end{enumerate} 
\end{Assumption}

Under Assumption \ref{jiashe1}, the state equation \eqref{state} admits a unique solution $y_t^u \in L_{\mathcal{F}}^2([0,T]\times\Omega;\mathbb{R})$ for any $\br{y^0,u} \in \mathbb{R}\times L^2([0,T]; \mathbb{R})$ (\cite{exm4}). For notational convenience, SOCP~\eqref{object} can be written into
\begin{eqnarray}\label{reduced}
    \mbox{Find}\ u^*\in U(\delta)\;\; \mbox{such that}\;\; J(u^*)=\min\limits_{u\in U(\delta)}J(u),
\end{eqnarray}
where $U(\delta)$ is the feasible control set given by
\begin{eqnarray}\label{con_U}
    U(\delta) := \left\{u\in L^2([0,T];\mathbb{R}) \;\vert\;G(u) \leq 0\right\}
\end{eqnarray}
with
\begin{eqnarray}\label{Gu}
    G(u) := \int_0^T\mathbb{E}\left[y_t^u\right] \di t - \delta.
\end{eqnarray}

To given the necessary condition of the optimal control, according to Chapter 1.7.3 of \cite{hinze}, the following regularity condition of Robinson is introduced:
\begin{Definition}(\cite{hinze})\label{robinson}
One says that the Robinson's regularity condition holds at $u(t)$ if
\begin{eqnarray}\label{Robin}
\begin{aligned}
0\in \textnormal{int}\left(G(u)+G^{\prime}(u)\left(L^2([0,T];\mathbb{R})\right)-K_{G}\right),
\end{aligned}
\end{eqnarray}
where $K_G=(-\infty,0]$, $G^{\prime}(u)\left(L^2([0,T];\mathbb{R})\right):=\left\{G^{\prime}(u)(v)|v\in L^2([0,T];\mathbb{R})\right\}$ and $G^{\prime}(u)(v)$ is the variation of $G(u)$ with respect to $u$ along the direction $v$.
\end{Definition}

Note that we can rewrite $G^{\prime}(u)(v)$ by introducing a BSDE. From the definition in (\ref{Gu}), for any $v\in L^2([0,T];\mathbb{R})$ we have
\begin{eqnarray}
\begin{aligned}\label{direction_G}
G^{\prime}(u)(v)&=\lim_{\kappa\rightarrow 0}\frac{G(u+\kappa v)-G(u)}{\kappa}
%=\lim_{\kappa\rightarrow 0}
%\frac{\mathbb{E}\left[\int_{0}^{T}y_t^{u+\kappa v}-y_t^{u}dt\right]}{\kappa}.
=\int_{0}^{T}\mathbb{E}[Dy_t^{u}(v)]dt,
\end{aligned}
\end{eqnarray}
where $t \mapsto Dy_t^{u}(v)$ is the variational process given by the following SDE:
\begin{eqnarray*}\left\{\begin{aligned}\label{dy_t}
dDy_t^{u}(v)&=\left(b_{y}^{\prime}(y_t^u,u)Dy_t^{u}(v)+b_{u}^{\prime}(y_t^u,u)v\right)dt
+\left(\sigma_{y}^{\prime}(y_t^{u},u)Dy_t^{u}(v)+\sigma_{u}^{\prime}(y_t^{u},u)v\right)dW_t,\\ Dy_0^{u}(v)&=0.
\end{aligned}\right.
\end{eqnarray*}
The existence of derivative in (\ref{direction_G}) has been discussed in \cite{state3,exist2,exist3} and applied in SOCP with control constraint \cite{max3}, data driven feedback control problem \cite{max2} and so on. To get rid of $Dy_t^{u}(v)$ in (\ref{direction_G}), we introduce a pair of adjoint process $(p_t^u, q_t^u)$, which is the adapted solution to the following BSDE:
\begin{eqnarray}\label{dp_t}
\left\{\begin{aligned}
-dp_{t}^{u}&=\left(1+p_{t}^{u}b_{y}^{\prime}(y_{t}^{u},u)+q_{t}^{u}\sigma_{y}^{\prime}(y_{t}^{u},u)\right)dt-q_{t}^{u}dW_t,\\
p_{T}^{u}&=0.
\end{aligned}\right.
\end{eqnarray}
By the It\^o formula, we obtain
\begin{eqnarray}\label{denote_dG}
\begin{aligned}
d\left(p_{t}^{u}Dy_t^{u}(v)\right)
&=Dy_t^{u}(v)dp_{t}^{u}+p_{t}^{u}dDy_t^{u}(v)+
q_{t}^{u}\left(\sigma_{y}^{\prime}(y_t^{u},u)Dy_t^{u}(v)+\sigma_{u}^{\prime}(y_t^{u},u)v\right)dt\\
%&=\left(-\left(1+p_{t}^{u}b_{y}^{'}(y_{t}^{u},u)+q_{t}^{u}\sigma_{y}^{'}(y_{t}^{u},u)\right)dt+q_{t}^{u}dW_t\right)Dy_t^{u}(v)\\
%&\quad+p_{t}^{u}\left(\left(b_{y}^{'}(y_t^u,u)Dy_t^{u}(v)+b_{u}^{'}(y_t^u,u)v\right)dt
%+\left(\sigma_{y}^{'}(y_t^{u},u)Dy_t^{u}(v)+\sigma_{u}^{'}(y_t^{u},u)v\right)dW_t \right)\\
%&\quad+q_{t}^{u}\left(\sigma_{y}^{'}(y_t^{u},u)Dy_t^{u}(v)+\sigma_{u}^{'}(y_t^{u},u)v\right)dt\\
&=-Dy_t^{u}(v)dt+\left(p_{t}^{u}b_{u}^{\prime}(y_t^u,u)v+q_{t}^{u}\sigma_{u}^{\prime}(y_t^{u},u)v \right)dt\\
&\quad+\left(q_{t}^{u}Dy_t^{u}(v)+p_{t}^{u}\left(\sigma_{y}^{\prime}(y_t^{u},u)Dy_t^{u}(v)+\sigma_{u}^{\prime}(y_t^{u},u)v\right)\right)dW_t,
\end{aligned}
\end{eqnarray}
where the second equality follows from (\ref{dp_t}). Integrating both sides of (\ref{denote_dG}) and taking the expectation yields
\begin{eqnarray*}\begin{aligned}\label{denote_G}
G^{\prime}(u)(v)
%=\mathbb{E}\left[\int_{0}^{T}Dy_t^{u}(v)dt\right]
=\int_{0}^{T}\left(\mathbb{E}\left[p_{t}^{u}b_{u}^{\prime}(y_t^u,u)+q_{t}^{u}\sigma_{u}^{\prime}(y_t^{u},u)\right]
\right)vdt.
\end{aligned}
\end{eqnarray*}

Now we are ready to present the necessary conditions of the optimal control for the SOCP \eqref{reduced}-\eqref{con_U}.
\begin{Theorem}\label{mu}
Let Assumption \ref{jiashe1} hold, and $u^*$ satisfied Robinson's regularity condition (\ref{Robin}) be the optimal control given in \eqref{reduced}, then there exists a real number $\mu^*\geq0$ and a pair of adjoint process $(p_t^*,q_t^*)$ such that
\begin{eqnarray}\label{con_first}\left\{\begin{aligned}
&dy_t^*=b(y_t^*,u^*)dt+\sigma(y_t^*,u^*)dW_t,\ y^*_0=y^0,\\
&-dp_t^*=\left(h^{\prime}(y_t^*)+p_t^*b_{y}^{\prime}(y_t^*,u^*)+q_t^*\sigma_{y}^{\prime}(y_t^*,u^*)+\mu^*\right)dt-q_t^*dW_t,\  
p_T^*=g(y_T^*)=k^{\prime}(y_T^*),\\
&\left[\mu^{*},w_t-y_t^*\right]\leq0,\ \forall w_t\in K,\\
%&\mu^*\mathbb{E}\left[\int_0^T(w_t-y_t^*)dt\right]\leq0,\ \forall w_t\in K,\\
&\mathbb{E}\left[p_t^*b_{u}^{\prime}(y_t^*,u^*)+q_t^* \sigma_{u}^{\prime}(y_t^*,u^*)\right]+j^{\prime}(u^*)=0.
\end{aligned}\right.
\end{eqnarray}
\end{Theorem}
\begin{proof}
By Theorem 1.56 in \cite{hinze} and the condition~\eqref{Robin}, there exists a Lagrange multiplier $\mu^*\geq0$ such that
\begin{eqnarray}\begin{aligned}\label{Lagrang_condi}
\mu^*G(u^{*})=0,\quad \mathcal{L}^{\prime}_{u}(u^*,\mu^*)(v-u^*)=0,\ \forall v\in L^2([0,T];\mathbb{R}),
\end{aligned}
\end{eqnarray}
where 
\begin{eqnarray}\begin{aligned}\label{Lagrangian}
\mathcal{L}(u^*,\mu^*):=J(u^*)+\mu^* G(u^*)
\end{aligned}
\end{eqnarray}
is the Lagrangian functional.
Inserting \eqref{Gu} into the first equality in \eqref{Lagrang_condi}, we obtain
\begin{equation}\label{eq_0eqMuGu}
0= \mu^*\br{[y^*, 1] - \delta} = \mu^* [y^* - w_t, 1] + \mu^* \br{[w_t, 1] - \delta}, \quad \forall w_t \in K
\end{equation}
with $K$ given in \eqref{state_con}.
Moreover, the definition of $K$ and $\mu^{*}\geq0$ imply that
\begin{equation*}
\mu^* \br{[w_t, 1] - \delta} \leq 0,
\end{equation*}
 and thus \eqref{eq_0eqMuGu} leads to
\begin{equation}\label{mu_budeng}
\mu^* [y^* - w_t, 1] \geq 0, \;\;\forall w_t \in K.
\end{equation}

From the definitions of $J(\cdot)$ and $G(\cdot)$, for any $v\in L^{2}([0,T];\mathbb{R})$ we have
\begin{eqnarray}\label{lag_der}\begin{aligned}
&\mathcal{L}_u^{\prime}(u^*,\mu^*)(v-u^*)\\
%&=J^{'}(u^*)(v-u^*)+\mu^* G^{'}(u^*)(v-u^*)\\
&=\lim_{\kappa\rightarrow 0}\frac{J\left(u^*+\kappa(v-u^*)\right)-J(u^*)}{\kappa}+\mu^*\lim_{\kappa\rightarrow 0}\frac{G\left(u^*+\kappa(v-u^*)\right)-G(u^*)}{\kappa}\\
&=\int_{0}^{T} \mathbb{E}\left[h^{\prime}(y_t^*)Dy_t^*(v-u^*)\right]dt+\mu^*\int_{0}^{T}\mathbb{E}\left[Dy_t^*(v-u^*)\right]dt \\
&\quad+\int_{0}^{T}j^{\prime}(u^*)(v-u^*)dt+\mathbb{E}\left[k^{\prime}(y_T^*)Dy_T^*(v-u^*) \right],
\end{aligned}
\end{eqnarray}
where $Dy_t^*(v-u^*)$ is the solution to the following SDE:
\begin{eqnarray*}
\left\{\begin{aligned}
dDy_t^{*}(v-u^{*})&=\left(b_{y}^{\prime}(y_t^*,u^*)Dy_t^{*}(v-u^*)+b_{u}^{\prime}(y_t^*,u^*)(v-u^*)\right)dt\\
&\quad+\left(\sigma_{y}^{\prime}(y_t^{*},u^*)Dy_t^{*}(v-u^*)+\sigma_{u}^{\prime}(y_t^{*},u^*)(v-u^{*})\right)dW_t,\\ Dy_0^{*}(v-u^{*})&=0.
\end{aligned}\right.
\end{eqnarray*}
In order to eliminate $Dy_t^{*}(v-u^{*})$ in (\ref{lag_der}), we introduce the following BSDE:
\begin{eqnarray}\label{adjoint}
\left\{\begin{aligned}
-dp_t^*&=f(y_t^*,p_t^*,q_t^*,u^*,\mu^*)dt-q_t^*dW_t,\\
p_T^*&=g(y_T^*)=k^{\prime}(y_T^*),
\end{aligned}\right.
\end{eqnarray}
where $f(y,p,q,u,\mu):=h^{\prime}(y)+pb_{y}^{\prime}(y,u)+q\sigma_{y}^{\prime}(y,u)+\mu$.
From the It\^{o} formula and (\ref{adjoint}), it follows that
\begin{eqnarray*}
\begin{aligned}
&d\left(Dy_t^{*}(v-u^{*})p_t^*\right)\\
&=Dy_t^{*}(v-u^{*})dp_t^*+p_t^*dDy_t^{*}(v-u^{*})+q_t^{*}\left(\sigma_{y}^{\prime}(y_t^{*},u^*)Dy_t^{*}(v-u^*)+\sigma_{u}^{\prime}(y_t^{*},u^*)(v-u^{*})\right)dt\\
&=(-h^{\prime}(y_t^*)-\mu^{*})Dy_t^{*}(v-u^{*})dt+\left(p_{t}^{*}b_{u}^{\prime}(y_t^{*},u^*)(v-u^*)
+q_{t}^{*}\sigma_{u}^{\prime}(y_t^{*},u^*)(v-u^*)\right)dt\\
&\quad+\left(q_t^{*}Dy_t^{*}(v-u^{*})+p_{t}^{*}\sigma_{y}^{\prime}(y_t^{*},u^*)Dy_t^{*}(v-u^*)+p_{t}^{*}\sigma_{u}^{\prime}(y_t^{*},u^*)(v-u^{*})\right)dW_t.
\end{aligned}
\end{eqnarray*}
Integrating both sides of the above equation, using the initial condition of $Dy_t^{*}(v-u^{*})$ and the terminal condition of $p_t^{*}$ and taking the expectation gives
\begin{eqnarray*}\begin{aligned}
&\mathcal{L}_u^{\prime}(u^*,\mu^*)(v-u^*)\\
&=\int_{0}^{T}\left(\mathbb{E}\left[p_t^*b_{u}^{\prime}(y_t^*,u^*)+q_t^* \sigma_{u}^{\prime}(y_t^*,u^*)\right]+j^{\prime}(u^*)\right)(v-u^*)dt=0,\ \forall v\in L^2([0,T];\mathbb{R}),
\end{aligned}
\end{eqnarray*}
which implies
\begin{eqnarray}\begin{aligned}\label{grad_lage}
\mathcal{L}^{\prime}_{u}(u^*,\mu^*)=\mathbb{E}\left[p_t^*b_{u}^{\prime}(y_t^*,u^*)+q_t^* \sigma_{u}^{\prime}(y_t^*,u^*)\right]+j^{\prime}(u^*)=0.
\end{aligned}
\end{eqnarray}
Combining (\ref{state}), (\ref{mu_budeng}), (\ref{adjoint}), (\ref{grad_lage}), the optimality condition system (\ref{con_first}) is obtained.
\end{proof}

Analogously, $J^{\prime}(u^{*})$ can be derived as
\begin{eqnarray}\label{grad_j}
\begin{aligned}
J^{\prime}(u^{*})&=\mathbb{E}\left[\hat{p}_t^{*}b_{u}^{\prime}(y_t^{*},u^{*})+\hat{q}_t^{*} \sigma_{u}^{\prime}(y_t^{*},u^{*})\right]+j^{\prime}(u^{*}).
\end{aligned}
\end{eqnarray}
Here $(\hat{p}_t^{*},\hat{q}_t^{*})$ is the solution to the following BSDE:
\begin{eqnarray*}\label{J_p}
\left\{\begin{aligned}
-d\hat{p}_t^{*}&=\hat{f}(y_t^{*},\hat{p}_t^{*},\hat{q}_t^{*},u^{*})dt-\hat{q}_t^{*}dW_t,\\ \hat{p}_T^{*}&=g(y_T^{*}),
\end{aligned}\right.
\end{eqnarray*}
where $\hat{f}(y,p,q,u):=f(y,p,q,u,\mu)-\mu$.

We close this section by the following lemma of projection.
\begin{Lemma}(\cite{max3})\label{pro_lemma}
Let $\mathbb{Q}$ be the projection operator from $L^2([0,T];\mathbb{R})$ onto a convex set $Q$ such that
\begin{eqnarray}\begin{aligned}\label{projector}
\|v-\mathbb{Q}v\|=\min_{z(t)\in Q}\|v-z\|.
\end{aligned}
\end{eqnarray}
Then $\mathbb{Q}v$ satisfies (\ref{projector}) if and only if, for any $z(t)\in Q$
\begin{eqnarray}\begin{aligned}\label{project_budengshi}
(\mathbb{Q}v-v,z-\mathbb{Q}v)\geq 0.
\end{aligned}
\end{eqnarray}
\end{Lemma}

\section{Gradient projection method and temporal discretization for control}
In this section, the gradient projection method is first presented with the following drift coefficient:
\begin{equation}\label{eq_linbyu}
    b(y, u) = b_{y}^{\prime}(t) y + b_{u}^{\prime}(t) u + m(t), \quad (t, y, u) \in [0, T] \times \R \times \R,
\end{equation}
where $b_{y}^{\prime}(t)$, $b_{u}^{\prime}(t)$ and $m(t)$ are the deterministic functions of $t$, and the derivatives of $b_{y}^{\prime}(t)$ and $b_{u}^{\prime}(t)$ with respect to $t$ are denoted by $b_{y,t}^{\prime\prime}(t)$ and $b_{u,t}^{\prime\prime}(t)$.
The central idea is to ensure that the state constraint holds by constructing specific multiplier at each iteration step. Then the step function is used to approximate the optimal control. 

\subsection{Gradient projection method}\label{GPM}
In the case where the drift coefficient $b$ satisfying \eqref{eq_linbyu}, $U(\delta)$ is a convex set and further it holds that
\begin{eqnarray*}
\begin{aligned} 
\left(J^{\prime}(u^{*}),v-u^{*}\right)\geq 0,
\end{aligned}
\end{eqnarray*}
which implies that 
\begin{eqnarray}\label{fix_point}
\begin{aligned} 
\left(u^{*}-\left(u^{*}-\rho J^{\prime}(u^{*})\right),v-u^{*}\right)\geq 0,
\end{aligned}
\end{eqnarray}
where $\rho$ is a positive constant.
From (\ref{project_budengshi}) in Lemma \ref{pro_lemma}, (\ref{fix_point}) implies that 
\begin{equation}\label{xian_ita}
\begin{aligned}
u^{*}=\mathbb{P}\big(u^{*}-\rho J^{\prime}(u^{*})\big),
\end{aligned}
\end{equation} 
where $\mathbb{P}$ is the projection operator from $L^2([0,T];\mathbb{R})$ onto the convex set $U(\delta)$ satisfied
\begin{equation}\label{define_pu}
\begin{aligned}
(\mathbb{P}w-w,v-\mathbb{P}w)\geq 0,\ \forall v\in U(\delta).
\end{aligned}
\end{equation}
Then, according to (\ref{xian_ita}), the following iterative scheme is given to approach the optimal control $u^{*}$ for a given initial value $u^{0,*}\in L^2([0,T];\mathbb{R})$:
\begin{equation}\label{ita_con}
\begin{aligned}
u^{i+1,*}=\mathbb{P}\big(u^{i,*}-\rho J^{\prime}(u^{i,*})\big),\ i=0,1,2,....
\end{aligned}
\end{equation}
Next the form of projection $\mathbb{P}$ is given and it is proven that for the iterative scheme (\ref{ita_con}) we can construct a explicit multiplier such that the $y_t^{i+1,*}$ solved by $u^{i+1,*}$ is within the constraint set.
\begin{Theorem}\label{mu_exist}
Under Assumption \ref{jiashe1}, for iterative scheme (\ref{ita_con}) there exist a explicit multiplier $\mu^{i,*}\geq0$ such that $\mathbb{E}\left[\int_{0}^{T}y_t^{i+1,*}dt\right]\leq \delta$ with a given $u^{i,*}\in L^2([0,T];\mathbb{R})$.
\end{Theorem}
\begin{proof}
We will establish the result in two steps: the first step gives the explicit forms of projection $\mathbb{P}$ as well as multiplier $\mu^{i,*}$ and proves that $\mathbb{E}\left[\int_{0}^{T}y_t^{i+1,*}dt\right]\leq \delta$; the second step proves that (\ref{define_pu}) is satisfied for projection $\mathbb{P}$. 

\textbf{Step 1}. For notational simplicity, we set $u^{i+\frac{1}{2},*}:=u^{i,*}-\rho J^{\prime}(u^{i,*})$ and the form of projection $\mathbb{P}$ is given by
\begin{equation}\label{pro_xian}
\begin{aligned}
\mathbb{P}\big(u^{i,*}-\rho J^{\prime}(u^{i,*})\big)=\mathbb{P}u^{i+\frac{1}{2},*}
=u^{i+\frac{1}{2},*}-\rho\mu^{i,*}\psi(t)b_u^{\prime}(t),
\end{aligned}
\end{equation}
where $\psi(t)$ satisfies the following ODE:
\begin{equation}\label{psi_xian}
\left\{\begin{aligned}
-d\psi(t)&=\left(1+\psi(t)b_y^{\prime}(t)\right)dt,\\ 
\psi(T)&=0.
\end{aligned}\right.
\end{equation}
We denote by $y_t^{i+\frac{1}{2},*}$ the state variable solved by $u^{i+\frac{1}{2},*}$ and $y_0^{i+\frac{1}{2},*}=y^0$. 
From (\ref{ita_con}) and (\ref{pro_xian}), it holds that $u^{i+1,*}=u^{i+\frac{1}{2},*}-\rho\mu^{i,*}\psi(t)b_u^{\prime}(t)$, which yields
\begin{eqnarray*}
\begin{aligned}
\mathbb{E}\left[y_t^{i+1,*}\right]=\mathbb{E}\left[y_t^{i+\frac{1}{2},*}\right]-\rho\mu^{i,*}\varphi(t),
\end{aligned}
\end{eqnarray*}
where $\varphi(t)$ satisfies
\begin{eqnarray}\label{varphi_con}
\left\{\begin{aligned}
d\varphi(t)&=b_y^{\prime}(t)\varphi(t)+\left(b_u^{\prime}(t)\right)^2\psi(t)dt,\\ 
\varphi(0)&=0.
\end{aligned}\right.
\end{eqnarray}
When the multiplier $\mu^{i,*}$ is chosen as
\begin{eqnarray}\label{def_mu}
\begin{aligned}
\mu^{i,*}=\frac{\max\left\{\mathbb{E}\left[\int_0^Ty_t^{i+\frac{1}{2},*}dt\right]-\delta,0\right\}}{\rho\int_{0}^{T}\varphi(t)dt},
\end{aligned}
\end{eqnarray}
it follows that $\mathbb{E}\left[\int_0^Ty_t^{i+1,*}dt\right]\leq\delta$.

\textbf{Step 2}. We verify that projection $\mathbb{P}$ given in (\ref{pro_xian}) satisfies (\ref{define_pu}).
For any $v\in U(\delta)$, we have
\begin{eqnarray*}\label{pro_con_1}
\begin{aligned}
&(u^{i+\frac{1}{2},*}-\mathbb{P}u^{i+\frac{1}{2},*},v-\mathbb{P}u^{i+\frac{1}{2},*})\\
&=\left(u^{i+\frac{1}{2},*}-(u^{i+\frac{1}{2},*}-\rho\mu^{i,*}\psi b_u^{\prime}),
v-(u^{i+\frac{1}{2},*}-\rho\mu^{i,*}\psi b_u^{\prime})    \right)\\
&=(\rho\mu^{i,*}\psi b_u^{\prime},
v-u^{i+\frac{1}{2},*}+\rho\mu^{i,*}\psi b_u^{\prime}  ).
\end{aligned}
\end{eqnarray*}
By the integration by parts formula and (\ref{psi_xian}), we arrive at
\begin{eqnarray}
\begin{aligned}\label{pro_con_2}
&(\rho\mu^{i,*}\psi b_u^{\prime},v-u^{i+\frac{1}{2},*})\\
%=\big(\rho\mu^{i,*}\psi(t),b_{u}^{'}(\hat{y}_t^{i+1,*},\hat{u}^{i+1,*})(v-\hat{u}^{i+1,*})\big)\\
&=\int_0^T\left(\rho\mu^{i,*}\psi(t)
\left(\frac{d\mathbb{E}\left[y_t^v-y_t^{i+\frac{1}{2},*}\right]}{dt}-b_{y}^{\prime}(t)
\mathbb{E}\left[y_t^{v}-y_t^{i+\frac{1}{2},*}\right] \right)   \right)dt\\
&=\int_0^T\rho\mu^{i,*}\psi(t)d\mathbb{E}\left[y_t^v-y_t^{i+\frac{1}{2},*}\right]
-\int_0^T\rho\mu^{i,*}\psi(t)b_{y}^{\prime}(t)\mathbb{E}\left[y_t^{v}-y_t^{i+\frac{1}{2},*}\right]dt\\
&=\int_0^T\rho\mu^{i,*}\mathbb{E}\left[y_t^v-y_t^{i+\frac{1}{2},*}
\right]\left(1+\psi(t)b_y^{\prime}(t)\right)dt
-\int_0^T\rho\mu^{i,*}\psi(t)b_{y}^{\prime}(t)\mathbb{E}\left[y_t^{v}-y_t^{i+\frac{1}{2},*}\right]dt\\
&=\int_0^T\rho\mu^{i,*}\mathbb{E}\left[y_t^v-y_t^{i+\frac{1}{2},*}\right]dt.
\end{aligned}
\end{eqnarray}
By ODEs (\ref{psi_xian}) and (\ref{varphi_con}), we obtain
\begin{eqnarray*}
\begin{aligned}
\int_0^T\psi(t)\left(b_y^{\prime}(t)\varphi(t)+\left(b_u^{\prime}(t)\right)^2
\psi(t) \right) dt
%=\int_0^T\frac{d\varphi(t)}{dt}\psi(t)dt
=\int_0^T\varphi(t)d\left(-\psi(t)\right)=\int_0^T\varphi(t)\left(1+\psi(t)b_y^{\prime}(t)\right)dt,
\end{aligned}
\end{eqnarray*}
which implies that
\begin{eqnarray}\label{pro_con_3}
\begin{aligned}
\int_0^T\left(\psi(t)b_u^{\prime}(t)\right)^2dt=\int_0^T\varphi(t)dt.
\end{aligned}
\end{eqnarray}
From (\ref{pro_con_2}), (\ref{pro_con_3}) and the definition of $\mu^{i,*}$ in (\ref{def_mu}), it holds that
\begin{eqnarray*}
\begin{aligned}
&(u^{i+\frac{1}{2},*}-\mathbb{P}u^{i+\frac{1}{2},*},v-\mathbb{P}u^{i+\frac{1}{2},*} )\\
&=\rho\mu^{i,*}\int_0^T\mathbb{E}\left[y_t^v\right]dt-\rho\mu^{i,*}\int_0^T\mathbb{E}\left[y_t^{i+\frac{1}{2},*}\right]dt+(\rho\mu^{i,*})^2\int_0^T\varphi(t)dt\\
&\leq\rho\mu^{i,*}\delta-\rho\mu^{i,*}\int_0^T\mathbb{E}\left[y_t^{i+\frac{1}{2},*}\right]dt
+\rho\mu^{i,*}\max\left\{\mathbb{E}\int_0^T\left[y_t^{i+\frac{1}{2},*}\right]dt-\delta,0\right\}\\
&=\rho\mu^{i,*}\left(\delta-\int_0^T\mathbb{E}\left[y_t^{i+\frac{1}{2},*}\right]dt+\max\left\{\int_0^T\mathbb{E}\left[y_t^{i+\frac{1}{2},*}\right]dt-\delta,0\right\}\right)
= 0.
\end{aligned}
\end{eqnarray*}
%which implies the definition of $\mathbb{P}$ satisfies satisfies (\ref{define_pu}). 
This completes the proof.
\end{proof}

We further have the following property for the projection $\mathbb{P}$.
\begin{Corollary}\label{xingzhi_P_con}
For the projection $\mathbb{P}$ given in (\ref{pro_xian}), it holds that
\begin{eqnarray*}
\begin{aligned}
\|\mathbb{P} w-\mathbb{P} z\|\leq\|w-z\|,
\end{aligned}
\end{eqnarray*}
for any $w,z\in L^2([0,T];\mathbb{R})$.
\end{Corollary}
\begin{proof}
Using (\ref{define_pu}) and Cauchy-Schwarz inequality yields the result.
%\end{eqnarray*}
\end{proof}
%From Theorem \ref{mu_exist}, the projection $\mathbb{P}$ defined in (\ref{pro_xian}) satisfies (\ref{define_pu}), which implies
%\begin{eqnarray*}
%\begin{aligned}
%\|\mathbb{P}w-\mathbb{P}z\|^2&=(\mathbb{P}w-\mathbb{P}z,\mathbb{P}w-\mathbb{P}z)\\
%&=(\mathbb{P}w-\mathbb{P}z,z-\mathbb{P}z)+(\mathbb{P}w-\mathbb{P}z,\mathbb{P}w-w)
%+(\mathbb{P}w-\mathbb{P}z,w-z)\\
%&\leq(\mathbb{P}w-\mathbb{P}z,w-z).
%\end{aligned}
%\end{eqnarray*}
%Then Cauchy-Schwarz inequality yields
%\begin{eqnarray*}
%\begin{aligned}
%\|\mathbb{P}w-\mathbb{P}z\|\leq\|w-z\|.
%\end{aligned}

\subsection{Temporal discretization for optimal control}
The optimal control $u^*$ is approximated by step function. For a positive integer $N$, a uniform time partition $\Pi=\{t_0,...,t_N\}$ over $[0,T]$ is introduced:
$$0=t_0<t_1<...<t_N=T,\ t_{n+1}-t_n=T/N=\Delta t.$$
We define the associated space of piecewise constant functions by
$$U_N=\left\{u\in L^2([0,T];\mathbb{R})|u=\sum\limits_{i=0}^{N-1}\beta_i\chi_{I_{i}^{N}}\ a.e.,\ \beta_i\in\mathbb{R}\right\},$$
where $I_{i}^{N}$ denotes the interval $[t_{i},t_{i+1})$ for $0\leq i\leq N-2$, and $I_{N-1}^{N}=[t_{N-1},t_{N}]$. $\chi_{I_{i}^{N}}(t)$ denotes the indicator function on $I_{i}^{N}$.

Let $U^{N}(\delta)=U_{N}\cap U(\delta)$, the approximated problem of (\ref{reduced})-(\ref{con_U}) is given by
\begin{eqnarray}\label{lisan_ocp}
\begin{aligned}
\mbox{Find}\ u^{*,N}\in U^N(\delta)\ \mbox{such}\ \mbox{that}\ J(u^{*,N})=\min\limits_{u^{N}\in U^N(\delta)}J(u^{N}).
\end{aligned}
\end{eqnarray}
Using similar argument as Theorem \ref{mu}, one can show that
\begin{eqnarray}\label{dis_fiest}\left\{
\begin{aligned}
&dy_t^{*,N}=b(y_t^{*,N},u^{*,N})dt+\sigma(y_t^{*,N},u^{*,N})dW_t,\ y^{*,N}_0=y^0,\\
%&-dp_t^{*,N}=\big(h^{'}(y_t^{*,N})+p_t^{*,N}b_{y}^{'}(y_t^{*,N},u^{*,N})+q_t^{*,N}\sigma_{y}^{'}(y_t^{*,N},u^{*,N})+\mu^{*,N}\big)dt-q_t^{*,N}dW_t,\ p_T^{*,N}=g(y_T^{*,N}),\\
&-dp_t^{*,N}=f(y_t^{*,N},p_t^{*,N},q_t^{*,N},u^{*,N},\mu^{*,N})dt-q_t^{*,N}dW_t,\
p_T^{*,N}=g(y_T^{*,N}),\\
&\left[\mu^{*,N},w_t-y_t^{*,N}\right]\leq0,\ \forall w_t\in K,\\
%&\mu^{*,N}\mathbb{E}\left[\int_0^T(w_t-y_t^{*,N})dt\right]\leq0,\ \forall w_t\in K,\\
&
%\mathcal{L}_{u}^{\prime}(u^{*,N},\mu^{*,N})= 
\mathbb{E}\left[p_t^{*,N}b_{u}^{\prime}(t)+q_t^{*,N} \sigma_{u}^{\prime}(y_t^{*,N},u^{*,N})\right]+j^{\prime}(u^{*,N})=0.
\end{aligned}\right.
\end{eqnarray}
Analogously, it holds that
\begin{eqnarray}\label{L_mu_dis}
\begin{aligned}
J^{\prime}(u^{*,N})=\mathbb{E}\left[\hat{p}_t^{*,N}b_{u}^{\prime}(t)+\hat{q}_t^{*,N} \sigma_{u}^{\prime}(y_t^{*,N},u^{*,N})\right]+j^{\prime}(u^{*,N}),
\end{aligned}
\end{eqnarray}
where
\begin{eqnarray}\label{J_p_dis}
\left\{\begin{aligned}
-d\hat{p}_t^{*,N}&=\hat{f}(y_t^{*,N},\hat{p}_t^{*,N},\hat{q}_t^{*,N},u^{*,N})dt-\hat{q}_t^{*,N}dW_t,\\ \hat{p}_T^{*,N}&=g(y_T^{*,N}).
\end{aligned}\right.
\end{eqnarray}

As discussed in subsection \ref{GPM}, for the SOCP (\ref{lisan_ocp}) with control discretization, we have the following iterative scheme for a given initial value $u^{0,N}\in U_{N}$:
\begin{equation}\label{ita_xian_dis}
\begin{aligned}
u^{i+1,N}
%=u^{i,N}-\mathcal{P}\big(\rho\mathcal{L}^{'}_{u}(u^{i,N},\mu^{i,N})\big)
=\mathbb{P}_N\big(u^{i,N}-\rho J^{\prime}(u^{i,N})\big),\ i=0,1,2,...,
\end{aligned}
\end{equation}
where $\rho>0$ is a constant and $\mathbb{P}_N$ is the projection from $L^2([0,T];\mathbb{R})$ to $U^N(\delta)$ satisfied 
\begin{equation*}\label{define_puN}
\begin{aligned}
(\mathbb{P}_Nw-w,v-\mathbb{P}_Nw)\geq 0,\ \forall v\in U^N(\delta).
\end{aligned}
\end{equation*}
For the projection $\mathbb{P}_N$, we have the following similar conclusions like Theorem \ref{mu_exist} and Corollary \ref{xingzhi_P_con}.
\begin{Theorem}\label{muh_exist}
Under Assumption \ref{jiashe1}, for iterative scheme (\ref{ita_xian_dis}) there exist a explicit multiplier $\mu^{i,N}\geq0$ such that $\mathbb{E}\left[\int_{0}^{T}y_t^{i+1,N}dt\right]\leq \delta$ with a given $u^{i,N}\in U_N$, where $y_t^{i+1,N}$ is the state variable solved by $u^{i,N}$ as the control variable. 
\end{Theorem}
\begin{proof}
%Here a concise proof is given. 
%\textbf{Step 1}. 
We set $u^{i+\frac{1}{2},N}:=u^{i,N}-\rho J^{\prime}(u^{i,N})$, $\hat{u}^{i+1,N}:= \mathcal{P}(u^{i+\frac{1}{2},N})$ and the form of projection $\mathbb{P}_N$ is given by
\begin{eqnarray}\label{define_P_dis}
\begin{aligned}
\mathbb{P}_N\big(u^{i,N}-\rho J^{\prime}(u^{i,N})\big)=\mathbb{P}_Nu^{i+\frac{1}{2},N}:=
\hat{u}^{i+1,N}-\mathcal{P}\big(\rho\mu^{i,N}\psi(t)b_u^{\prime}(t)\big),
\end{aligned}
\end{eqnarray}
where $\mathcal{P}$ is the $L^2$-projection operator from $L^2([0,T];\mathbb{R})$ onto $U_N$ and it is defined by
\begin{equation*}\label{piecewise_con}
\begin{aligned}
(\mathcal{P}v,w)=(v,w),\ \forall w\in U_N,\ v\in L^2([0,T];\mathbb{R}).
\end{aligned}
\end{equation*}
For the projection $\mathbb{P}_N$ given in (\ref{define_P_dis}), the following $\mu^{i,N}$ is given
\begin{eqnarray}\label{def_mu_N}
\begin{aligned}
\mu^{i,N}=\frac{\max\left\{\mathbb{E}\left[\int_0^T\hat{y}_t^{i+1,N}dt\right]-\delta,0\right\}}{\rho\int_{0}^{T}\tilde{\varphi}(t)dt},
\end{aligned}
\end{eqnarray}
where $\hat{y}_t^{i+1,N}$ is solved by $\hat{u}^{i+1,N}$ and
$\tilde{\varphi}(t)$ satisfies
\begin{eqnarray}
\left\{\begin{aligned}\label{varphi_dis}
d\tilde{\varphi}(t)&=b_y^{\prime}(t)\tilde{\varphi}(t)+b_u^{\prime}(t)
\mathcal{P}\big(b_u^{\prime}(t)\psi(t)\big)dt,\\ 
\tilde{\varphi}(0)&=0.
\end{aligned}\right.
\end{eqnarray}
Then the conclusion can be proved as in Theorem \ref{mu_exist}.
\end{proof}

Analogously, we have the following conclusion.
\begin{Corollary}\label{xingzhi_P_dis}
For the operator $\mathbb{P}_N$ given in (\ref{define_P_dis}), it holds that
\begin{eqnarray*}
\begin{aligned}
\|\mathbb{P}_N w-\mathbb{P}_N z\|\leq\|w-z\|,
\end{aligned}
\end{eqnarray*}
for any $w,z\in L^2([0,T];\mathbb{R})$.
\end{Corollary}

For notational simplicity, we will omit superscripts in the following analysis, e.g., $u=u^{*,N},y_t=y_t^{*,N},\mu=\mu^{*,N}$, etc.
Under mild assumptions, the backward equation (\ref{J_p_dis}) is well-posed \cite{max3,kac}. Moreover, by the nonlinear Feynman-Kac formula, the solution $(\hat{p}_t,\hat{q}_t)$ has the representations
\begin{eqnarray}\label{kac}
\begin{aligned}
\hat{p}_t=\hat{\eta}(t,y_t),\quad \hat{q}_t=\sigma\left(y_t,u(t)\right)\hat{\eta}_{y}^{\prime}(t,y_t).
\end{aligned}
\end{eqnarray}
Here
$\hat{\eta}(t,y):[0,T]\times\mathbb{R}\rightarrow\mathbb{R}$ is the solution of the following parabolic PDE:
\begin{eqnarray*}
\left\{\begin{aligned}
\mathcal{L}^0\hat{\eta}(t,y)&=-\hat{f}\left(y,\hat{\eta}(t,y),\sigma\left(y,u(t)\right)\hat{\eta}_{y}^{\prime}(t,y),u(t)\right),\\
\hat{\eta}(T,y)&=g(y),
\end{aligned}\right.
\end{eqnarray*}
where $\mathcal{L}^0\hat{\eta}(t,y):=\hat{\eta}_{t}^{\prime}(t,y)+b\left(y,u(t)\right)\hat{\eta}_{y}^{\prime}(t,y)+
\frac{1}{2}\sigma\left(y,u(t)\right)^2\hat{\eta}_{y,y}^{\prime\prime}(t,y).$
From the definitions of $\mathcal{L}_{u}^{\prime}(u,\mu)$ and $J^{\prime}(u)$, it can be derived that 
\begin{eqnarray}\label{relation}
\begin{aligned} 
p_t=\hat{p}_t+\mu\psi(t),\quad q_t=\hat{q}_t.
\end{aligned}
\end{eqnarray}
%which implies that 
%\begin{eqnarray}\label{fam_kac}
%\begin{aligned}
%p_t=\hat{\eta}(t,y_t)+\mu\psi(t),\quad q_t=\sigma\left(y_t,u(t)\right)\hat{\eta}_{y}^{\prime}(t,y_t).
%\end{aligned}
%\end{eqnarray}
Note that $J^{\prime}(\cdot)$ is composed of the solutions of FBSDEs, which implies iterative scheme (\ref{ita_xian_dis}) is still not feasible in practice and we need to further numerically approximate the FBSDEs to get fully computable approximation $J_{N}^{\prime}(\cdot)$ of $J^{\prime}(\cdot)$. %(\ref{fam_kac}) indicates that both $p_t$ and $\hat{p_t}$ are functions with respect to $y_t$, which is an important information for approximating them in the next section.

\section{Fully discrete iterative scheme and error estimates}
In this section, the discussion is still carried out under the linear drift coefficient. We first discretize FBSDEs numerically to get fully discrete scheme, where Euler method is used in time and the least square Monte Carlo method is used to approximate the conditional expectations, respectively. Then the fully discrete iterative scheme and algorithm are obtained. Further, the error estimates for the control and multiplier are deduced.

\subsection{Fully discrete iterative scheme}\label{lsmc_process}
Following \cite{goblt2,goblt3,other1,other3}, decoupling the BSDE in (\ref{dis_fiest}) by taking the conditional expectation $\mathbb{E}\left[\cdot|\mathcal{F}_{t_n}\right]$ yields the following semi-discrete scheme:
\begin{equation}\label{disscheme}\left\{
\begin{aligned}
&y_0=y^0,\ p_{N}=g(y_N),\\
&y_{n+1}= y_{n}+b\left(y_{n},u(t_n)\right)\Delta t+\sigma\left(y_{n},u(t_n)\right)\Delta W_{n+1},\ \ n=0,...,N-1,\\
&q_{n}=\frac{1}{\Delta t}\mathbb{E}[\Delta W_{n+1}p_{n+1}|\mathcal{F}_{t_n}],\ n=N-1,...,0,\\
&p_{n}
=\mathbb{E}\big[p_{n+1}+f\left(y_{n},p_{n+1},q_{n},u(t_n),\mu_h\right)\Delta t|\mathcal{F}_{t_n}\big],\ n=N-1,...,0.
\end{aligned}\right.
\end{equation}
A similar analysis of (\ref{J_p_dis}) yields
\begin{equation}\label{disscheme_fu}\left\{
\begin{aligned}
%\quad q_{t_N}^{\pi}=0,\\
\hat{q}_{n}&=\frac{1}{\Delta t}\mathbb{E}[\Delta W_{n+1}\hat{p}_{n+1}|\mathcal{F}_{t_n}],\ \hat{p}_{N}=g(y_N),\\
\hat{p}_{n}
&=\mathbb{E}\big[\hat{p}_{n+1}+\hat{f}\left(y_{n},\hat{p}_{n+1},\hat{q}_{n},u(t_n)\right)\Delta t|\mathcal{F}_{t_n}\big].
\end{aligned}\right.
\end{equation}
To maintain the relationship between $(p_n,q_n)$ and $(\hat{p}_n,\hat{q}_n)$ as in (\ref{relation}), the following discrete scheme is applied to (\ref{psi_xian}):
\begin{eqnarray}\label{psi_right}
\begin{aligned}
\psi_n&=\psi_{n+1}+\left(1+\psi_{n+1}b_y^{\prime}(t_n)\right)\Delta t,\ \psi_N=0,\ n=N-1,...,1,0.
\end{aligned}
\end{eqnarray}
Then the relationship between $(p_n,q_n)$ and $(\hat{p}_n,\hat{q}_n)$ is deduced as follows:
\begin{equation}\label{markov}
\begin{aligned}
p_n&=\hat{p}_n+\mu_h\psi_n,\ n=N,...,1,0,\quad q_n=\hat{q_n},\ n=N-1,...,0.
\end{aligned}
\end{equation}
We perform constant interpolation for them to get $\psi^{\pi}(t)$, $y_t^{\pi}$, $(\hat{q}_{t}^{\pi},\hat{p}_{t}^{\pi})$ and $(q_{t}^{\pi},p_{t}^{\pi})$, such as $\psi^{\pi}(t)=\sum_{n=0}^{N-1}\psi_{n}\chi_{I_{n}^{N}}(t)$.
%Note that the discrete schemes $(\ref{disscheme})$ and $(\ref{disscheme_fu})$ all contain conditional expectations. 
In practice, we need to further approximate the conditional expectations in $(\ref{disscheme})$ and $(\ref{disscheme_fu})$, and the LSMC method is used to accomplish this.
As $p_{N}$ and $\hat{p}_N$ are the deterministic function of $y_{N}$ and $(y_{n},\mathcal{F}_{t_n})_{t_n\in\Pi}$ is a Markov chain, which means that
\begin{equation}\label{markov_fu}
\begin{aligned}
&p_N=g(y_{N})=P_N(y_N),\ p_n=P_n(y_n),\ q_n=Q_n(y_n),\\
&\hat{p}_N=g(y_{N})=\hat{P}_N(y_N),\ \hat{p}_n=\hat{P}_n(y_n),\ \hat{q}_n=\hat{Q}_n(y_n),\\
%&\psi_N=0(y_{N})=\Psi_N(y_N),\ \psi_n=\Psi(y_n),\ \hat{\psi}_n=\hat{\Psi}_n(y_n),\ n=N-1,...,0,
\end{aligned}
\end{equation}
where $n=N-1,...,0$ and $P_{n}(\cdot),Q_{n}(\cdot)$, $\hat{P}_{n}(\cdot),\hat{Q}_{n}(\cdot)$ are the unknown regression functions. From (\ref{markov}) and (\ref{markov_fu}), we have
\begin{equation*}\label{bijin1}
\begin{aligned}
P_n(y_n)=\hat{P}_n(y_n)+\mu_h\psi_n,\ n=N,...,1,0,\ 
Q_n(y_n)=\hat{Q}_n(y_n),\ n=N-1,...,1,0.
\end{aligned}
\end{equation*}
From the elementary property of the conditional projection, for $n=N-1,...,0$, it holds that
\begin{equation}\left\{\label{cond_exp}
\begin{aligned}
{Q}_{n}(y_{n})
&=\mbox{arg}\min\limits_{v_{0,n}(y_{n})}\mathbb{E}\big[|\frac{\Delta W_{n+1}}{\Delta t}{P}_{n+1}(y_{n+1})-v_{0,n}(y_{n})|^2\big],\\
{P}_{n}(y_{n})
&=\mbox{arg}\min\limits_{v_{1,n}(y_{n})}\mathbb{E}\big[|{P}_{n+1}(y_{n+1})
+{f}\big(y_{n},{P}_{n+1}(y_{n+1}),{Q}_{n}(y_{n}),u(t_n),\mu_h\big)\Delta t \\
&\quad-v_{1,n}(y_{n})|^2\big],
\end{aligned}\right.
\end{equation}
where $v_{0,n}(y_n)$ and $v_{1,n}(y_n)$ run over all $\mathcal{F}_{t_n}$  measurable functions with
$\mathbb{E}\left[|v_{0,n}(y_{n})|^2\right]<+\infty$ 
and $\mathbb{E}\left[|v_{1,n}(y_{n})|^2\right]<+\infty$.
In order to solve this infinite-dimensional minimization problem, we first fix the finite dimensional subspaces
\begin{equation*}
\begin{aligned}
\Lambda_{0,n}=\mbox{span}\{\boldsymbol{\eta}^n\}=\mbox{span}\{\eta_1^n,...,\eta_K^n\},\ \Lambda_{1,n}=\mbox{span}\{\boldsymbol{\xi}^n\}=\mbox{span}\{\xi_1^n,...,\xi_{\tilde{K}}^n\},
\end{aligned}
\end{equation*}
%of $\mathcal{F}_{t_n}$ measurable and square-integrable space 
for $n=0,1,...,N-1$. In principle, the number of basis functions of $\Lambda_{0,n}$ and $\Lambda_{1,n}$ can be different at each time step $t_n$, which we suppress for simplicity.
Then (\ref{cond_exp}) is approximated to: find $\boldsymbol{\alpha}_{0,n}^{K}\in\mathbb{R}^{K}$ and $\boldsymbol{\alpha}_{1,n}^{\tilde{K}}\in\mathbb{R}^{\tilde{K}}$ for $n=N-1,...,0$ such that
\begin{equation}\label{infinite}\left\{
\begin{aligned}
&{P}_{N}^{K}(y_{N})={P}_{N}(y_{N}),\\
&\boldsymbol{\alpha}_{1,n}^{\tilde{K}}=\mbox{arg}\min\limits_{\boldsymbol{\tilde{\alpha}}_{1,n}^{\tilde{K}}\in\mathbb{R}^{\tilde{K}}}
\mathbb{E}\big[|\frac{\Delta W_{n+1}}{\Delta t}{P}_{n+1}^{K}(y_{n+1})-\boldsymbol{\xi}^n(y_{n})\boldsymbol{\tilde{\alpha}}_{1,n}^{\tilde{K}}|^2\big],\\
&{Q}_{n}^{\tilde{K}}(y_{n})=\boldsymbol{\xi}^n(y_{n})\boldsymbol{\alpha}_{1,n}^{\tilde{K}},\\
&\boldsymbol{\alpha}_{0,n}^{K}=\mbox{arg}\min\limits_{\boldsymbol{\tilde{\alpha}}_{0,n}^{K}\in\mathbb{R}^{K}}
\mathbb{E}\big[|{P}_{n+1}^{K}(y_{n+1})
+{f}\big(y_{n},{P}_{n+1}^{K}(y_{n+1}),{Q}_{n}^{\tilde{K}}(y_{n}),u(t_n),\mu_{h,K}\big)\Delta t \\
&\quad\quad\quad-\boldsymbol{\eta}^n(y_{n})\boldsymbol{\tilde{\alpha}}_{0,n}^{K}|^2\big],\\
&{P}_{n}^{K}(y_{n})=\boldsymbol{\eta}^n(y_{n})\boldsymbol{\alpha}_{0,n}^{K}.\\
\end{aligned}\right.
\end{equation}
%For the sake of clarity, we denote $\left({Q}_{n}^{\tilde{K}}(y_{n}),{P}_{n}^{K}(y_{n})\right)$ as follows:
%\begin{equation*}
%\left\{\begin{aligned}
%{Q}_{n}^{\tilde{K}}(y_{n})&=\mathcal{P}_{1,n}\left(\frac{\Delta W_{n+1}}{\Delta t}{P}_{n+1}^{K}(y_{n+1})\right),\\
%{P}_{n}^{K}(y_{n})&=\mathcal{P}_{0,n}\left({P}_{n+1}^{K}(y_{n+1})
%+{f}\left(y_{n},{P}_{n+1}^{K}(y_{n+1}),{Q}_{n}^{\tilde{K}}(y_{n}),u(t_n),\mu_{h,K}\right)\Delta t\right).
%\end{aligned}\right.
%\end{equation*}
%The definitions of $\mathcal{P}_{1,n}(\cdot)$ and $\mathcal{P}_{0,n}(\cdot)$ can be directly obtained by (\ref{infinite}). It follows that $\mathcal{P}_{1,n}(\cdot)$ and $\mathcal{P}_{0,n}(\cdot)$ are actually the orthogonal projections onto spaces $\Lambda_{0,n}$ and $\Lambda_{1,n}$, and they are used to approximate the conditional expectations in (\ref{disscheme}).
The least square approximation $\big(P_{n}^{K}(y_{n}),Q_{n}^{\tilde{K}}(y_{n})\big)$ of $(p_n,q_n)$ is obtained. Similarly,
we can obtain the least square approximation $\big(\hat{P}_{n}^{K}(y_{n}),\hat{Q}_{n}^{\tilde{K}}(y_{n})\big)$ of $(\hat{p}_n,\hat{q}_n)$ by performing (\ref{infinite}).
Then we further use Monte Carlo method to approximate the expectations in (\ref{infinite}).
The $L$ independent copies $(\Delta_\lambda W_{n+1},_{\lambda}{y_{n+1}})_{n=0,...,N-1}^{\lambda=1,...,L}$ of $\Delta W_{n+1}$ and $y_{n+1}$ are given and we arrive at
\begin{equation}\label{MC1}\left\{
\begin{aligned}
&{P}_{N}^{K,L}(y_{N})={P}_{N}^{K}(y_{N}),\\
&\boldsymbol{\alpha}_{1,n}^{\tilde{K},L}=\mbox{arg}\min\limits_{\boldsymbol{\tilde{\alpha}}_{1,n}^{\tilde{K},L}\in\mathbb{R}^{\tilde{K}}}\frac{1}{L}\sum\limits_{\lambda=1}^{L}
\big(\frac{\Delta_{\lambda}W_{n+1}}{\Delta t}{P}_{n+1}^{K,L}(_{\lambda}y_{n+1})-\boldsymbol{\xi}^n(_{\lambda}y_{n})\boldsymbol{\tilde{\alpha}}_{1,n}^{\tilde{K},L}\big)^2,\\
&{Q}_{n}^{\tilde{K},L}(y_{n})=\boldsymbol{\xi}^n(y_{n})\boldsymbol{\alpha}_{1,n}^{\tilde{K},L},\\
&\boldsymbol{\alpha}_{0,n}^{K,L}=\mbox{arg}\min\limits_{\boldsymbol{\tilde{\alpha}}_{0,n}^{K,L}\in\mathbb{R}^{K}}
\frac{1}{L}\sum\limits_{\lambda=1}^{L}\big({P}_{n+1}^{K,L}(_{\lambda}y_{n+1})\\
&\quad\quad\quad+{f}\big(_{\lambda}y_{n},{P}_{n+1}^{K,L}(_{\lambda}y_{n+1}),{Q}_{n}^{\tilde{K},L}(_{\lambda}y_{n}),u(t_n),\mu_{h,KL}\big)\Delta t -\boldsymbol{\eta}^n(_{\lambda}y_{n})\boldsymbol{\tilde{\alpha}}_{0,n}^{K,L}\big)^2,\\
&{P}_{n}^{K,L}(y_{n})=\boldsymbol{\eta}^n(y_{n})\boldsymbol{\alpha}_{0,n}^{K,L}.
\end{aligned}\right.
\end{equation}
If the basic functions are given, the LSMC approximation $\big({P}_{n}^{K,L}(y_{n}),{Q}_{n}^{\tilde{K},L}(y_{n})\big)$ of $({p}_{n},{q}_{n})$ can be obtained. Performing the operation (\ref{MC1}) for 
$\big(\hat{P}_n^{K}(y_n),\hat{Q}_n^{\tilde{K}}(y_n)\big)$ yields 
$\big(\hat{P}_n^{K,L}(y_n),\hat{Q}_n^{\tilde{K},L}(y_n)\big)$, which is the LSMC approximation of $(\hat{p}_n,\hat{q}_n)$.
After the constant interpolations for $\big(P_{n}^{K}(y_n),Q_{n}^{\tilde{K}}(y_n)\big)$ and $\big(P_{n}^{K,L}(y_n),Q_{n}^{\tilde{K},L}(y_n)\big)$, the terms
$(P_{t}^{K,\pi},Q_{t}^{\tilde{K},\pi})$ and $(P_{t}^{K,L,\pi},Q_{t}^{\tilde{K},L,\pi})$ are obtained, such as
$P_{t}^{K,\pi}=\sum_{n=0}^{N-1}{P}_{n}^{K}(y_n)\chi_{I_{n}^{N}}(t)$.
Doing the same operation yields $(\hat{P}_{t}^{K,\pi},\hat{Q}_{t}^{\tilde{K},\pi})$ and $(\hat{P}_{t}^{K,L,\pi},\hat{Q}_{t}^{\tilde{K},L,\pi})$, which are some stochastic processes.

In order to obtain fully computable $J_{N}^{\prime}(\cdot)$, we approximate the expectation with $L$ Monte Carlo simulation paths given in (\ref{MC1}). 
First we define 
$_{\lambda}y_t^{\pi}:=\sum_{n=0}^{N}{_{\lambda}y_n}\chi_{I_{n}^{N}}(t)$, $\lambda=1,...,L,$
which is an $L$-dimensional vector at any $t\in(0,T]$.
%where $_{\lambda}y_0$ is viewed as an $L$-dimensional vector with values all $y^0$. 
Analogously, the Monte Carlo simulations for 
$(\hat{P}_{t}^{K,L,\pi},\hat{Q}_{t}^{\tilde{K},L,\pi})$ 
and $(P_{t}^{K,L,\pi},Q_{t}^{\tilde{K},L,\pi})$ are denoted by
$(_{\lambda}\hat{P}_{t}^{K,L,\pi},_{\lambda}\hat{Q}_{t}^{\tilde{K},L,\pi})$ and $(_{\lambda}P_{t}^{K,L,\pi},_{\lambda}Q_{t}^{\tilde{K},L,\pi})$, such as $_{\lambda}\hat{P}_{t}^{K,L,\pi}:=\sum_{n=0}^{N-1}\hat{P}_{n}^{K,L}(_{\lambda}y_n)\chi_{I_{n}^{N}}(t)$,
which are all $L$-dimensional vectors for a given time $t$. Based on the relation between $(\hat{p}_t,\hat{q}_t)$ and ${J}^{\prime}(u)$ in (\ref{L_mu_dis}), we obtain naturally the following definition of ${J}_N^{'}(u)$:
\begin{equation}\label{fully_lj_mc}
\begin{aligned}
J_{N}^{\prime}(u):=\frac{1}{L}\sum_{\lambda=1}^{L}\big[_{\lambda}\hat{P}_t^{K,L,\pi}
\sum_{n=0}^{N-1}b_{u}^{\prime}(t_n)\chi_{I_{n}^{N}}(t)+_{\lambda}\hat{Q}_{t}^{\tilde{K},L,\pi}\sigma_{u}^{
\prime}(_{\lambda}y_{t}^{\pi},u)\big]+j^{\prime}(u).
\end{aligned}
\end{equation}
Then the fully discrete iterative scheme based on LSMC method is obtained as follows for a given initial value $u^{0,N}\in U_{N}$ and $i=0,1,2,...$:
\begin{equation}\label{fully_ita_lisan2}
\begin{aligned}
u^{i+1,N}
=\mathbb{P}_N\big(u^{i,N}-\rho J_N^{\prime}(u^{i,N})\big)
=u^{i,N}-\rho J_N^{\prime}(u^{i,N})
-\rho\mu_{h,KL}^{i,N}\sum_{n=0}^{N-1}\psi_{n}b_{u}^{\prime}(t_n)\chi_{I_{n}^{N}}(t).
\end{aligned}
\end{equation}
The following gradient projection algorithm based on LSMC discretization is given:

\begin{algorithm}[!htbp]
 \SetKwData{Left}{left}\SetKwData{This}{this}\SetKwData{Up}{up} \SetKwFunction{Union}{Union}\SetKwFunction{FindCompress}{FindCompress} \SetKwInOut{Input}{Require}\SetKwInOut{Output}{Ensure}
\Input{Constant $\rho>0$, initial value of the control $u^{0,N}\in U_N$, error tolerance $\varepsilon_0$ and $L$ independent copies $\{\Delta_{\lambda}W_{n+1}\}_{n=0,...,N-1}^{\lambda=1,2,...,L}$ of $\{\Delta W_{n+1}\}_{n=0,...,N-1}$.} 
\Output{}
\BlankLine 

\emph{Set $error>\varepsilon_0$, $i=1$}\;
\While {$error>\varepsilon_0$}
{
Solving the state equation by $u^{i-1,N}$ and $\{\Delta_{\lambda}W_{n+1}\}_{n=0,...,N-1}^{\lambda=1,2,...,L}$ to get $_{\lambda}y_{t}^{\pi}$\;
Solving the ODE in (\ref{psi_xian}) and BSDE in (\ref{J_p_dis}) by performing (\ref{psi_right}) and (\ref{MC1}) to get $(_{\lambda}\hat{P}_{t}^{K,L,\pi},_{\lambda}\hat{Q}_{t}^{\tilde{K},L,\pi})$ and $\psi_{n}$\;
Computing $\hat{u}^{i,N}=u^{i-1,N}-\mathcal{P}\big(\rho J_{N}^{\prime}(u^{i-1,N})\big)$ by using the definition in (\ref{fully_lj_mc})\;
Solving the state equation by $\hat{u}^{i,N}$ and the forward Euler method to get $\{_{\lambda}\hat{y}_{n+1}^{i,N}\}_{n=0,1,...,N-1}^{\lambda=1,2,...,L}$\;
Computing $\hat{I}_i:=\mathbb{E}\left[\int_{0}^{T}\hat{y}_{t}^{i,N}dt\right]$ numerically by using $\{_{\lambda}\hat{y}_{n+1}^{i,N}\}_{n=0,1,...,N-1}^{\lambda=1,2,...,L}$ and trapezoidal rule\;
Solving (\ref{varphi_dis}) by the forward Euler method to get $\tilde{\varphi}_{n}, n=0,1,...,N$\;
Computing $\tilde{I}_i:=\int_{0}^{T}\tilde{\varphi}(t)dt$ numerically by using $\tilde{\varphi}_{n}$ and trapezoidal rule\;
Computing $\mu_{h,KL}^{i-1,N}=\frac{\max\{\hat{I}_i-\delta,0\}}{\rho \tilde{I}_i}$\;
Updating $u^{i,N}$ by: $u^{i,N}=\hat{u}^{i,N}-\rho \mu_{h,KL}^{i-1,N}\sum_{n=0}^{N-1}\psi_nb_{u}^{\prime}(t_n)\chi_{I_{n}^{N}}(t)$\;
Computing $error=\|u^{i,N}-u^{i-1,N}\|_{\infty}$, and let $i=i+1$.
}
\caption{Gradient projection algorithm with LSMC method}
\label{lsmc_algorithm_xian} 
\end{algorithm}

\subsection{Error estimates}\label{error_grid}
In this subsection, the error estimates for control and multiplier are derived for the fully discrete iterative scheme (\ref{fully_ita_lisan2}). Before that, the following assumption and estimate for discrete scheme (\ref{psi_right}) are given:
\begin{Assumption}\label{jiashe2}
We assume $0<c_b\leq|b_{u}^{\prime}|$ and the derivatives of deterministic functions $b_y^{\prime}(t)$, $b_u^{\prime}(t)$ are bounded, i.e., there exist a positive constant $C_t$ such that
$|b_{y,t}^{\prime\prime}|+|b_{u,t}^{\prime\prime}|\leq C_t.$
\end{Assumption}

\begin{Theorem}\label{psi_theorem}
Under Assumptions \ref{jiashe1} and \ref{jiashe2}, for the discrete scheme (\ref{psi_right}) of $\psi(t)$ and sufficiently small $\Delta t$, it holds that
\begin{eqnarray*}
\begin{aligned}
\max_{0\leq n\leq N}|\psi_n-\psi(t_n)|\leq C\Delta t, 
\end{aligned}
\end{eqnarray*}
where $C$ is a positive constant independent of $N$.
\end{Theorem}

\begin{proof}
The explicit Euler discrete scheme is applied to (\ref{psi_xian}), i.e.,
\begin{eqnarray}\label{yinshi}
\begin{aligned}
\tilde{\psi}_n=\tilde{\psi}_{n+1}+\left(1+\tilde{\psi}_{n+1}b_y^{\prime}(t_{n+1})\right)\Delta t,\ \tilde{\psi}_N=0,\ n=N-1,...,1,0.
\end{aligned}
\end{eqnarray}
We have the standard estimate as follows for sufficiently small $\Delta t$:
\begin{eqnarray}\label{psi_xianshi}
\begin{aligned}
\max_{0\leq n\leq N-1}|\tilde{\psi}_n-\psi(t_n)|\leq C\Delta t.
\end{aligned}
\end{eqnarray}
By differentiating (\ref{yinshi}) from (\ref{psi_right}) and using Taylor expansion, it holds that
\begin{eqnarray}\label{psi_right_esti}
\begin{aligned}
|\tilde{\psi}_n-\psi_n|
&=|\tilde{\psi}_{n+1}-\psi_{n+1}+\tilde{\psi}_{n+1}b_{y}^{\prime}(t_{n+1})\Delta t-\psi_{n+1}b_{y}^{\prime}(t_n)\Delta t|\\
&=|\tilde{\psi}_{n+1}-\psi_{n+1}+(\tilde{\psi}_{n+1}-\psi_{n+1})b_{y}^{\prime}(t_n)\Delta t+
\tilde{\psi}_{n+1}b_{y,t}^{\prime\prime}(t_n+\theta\Delta t)(\Delta t)^2|\\
&\leq(1+C_b\Delta t)|\tilde{\psi}_{n+1}-\psi_{n+1}|+C_t(\Delta t)^2\big(|\psi(t_{n+1})|+C\Delta t \big)\\
&\leq(1+C_b\Delta t)|\tilde{\psi}_{n+1}-\psi_{n+1}|+C_t(\Delta t)^2\big(\frac{e^{C_bT}-1}{C_b}+C\Delta t\big),
\end{aligned}
\end{eqnarray}
where $\theta\in(0,1)$. Employing the recursion formula above yields
\begin{eqnarray*}
\begin{aligned}
|\tilde{\psi}_n-\psi_n|&\leq(1+C_b\Delta t)^{N-n-1}|\tilde{\psi}_{N-1}-\psi_{N-1}|\\
&\quad+\sum\limits_{j=0}^{N-n-2}(1+C_b\Delta t)^{j}C_t(\Delta t)^2\big(\frac{e^{C_bT}-1}{C_b}+C\Delta t\big).
\end{aligned}
\end{eqnarray*}
By the termination condition $\tilde{\psi}_N=\psi_N=0$ and (\ref{psi_right_esti}), it holds that
$|\tilde{\psi}_{N-1}-\psi_{N-1}|=0$,
which implies
\begin{eqnarray}\label{tildepsi}
\begin{aligned}
|\tilde{\psi}_n-\psi_n|&\leq\sum\limits_{j=0}^{N-n-2}(1+C_b\Delta t)^{j}C_t(\Delta t)^2\big(\frac{e^{C_bT}-1}{C_b}+C\Delta t\big)\\
&=\big((1+C_b\Delta t)^{N-n-1}-1\big)C_t\Delta t\big(\frac{e^{C_bT}-1}{C_b^2}+\frac{C\Delta t}{C_b}\big)\\
&\leq\big((1+\frac{C_b}{N})^{N}-1\big)C_t\Delta t\big(\frac{e^{C_bT}-1}{C_b^2}+\frac{C\Delta t}{C_b}\big)\\
&= (e^{C_b}-1)C_t\big(\frac{e^{C_bT}-1}{C_b^2}+\frac{C\Delta t}{C_b}\big)\Delta t,\quad N\rightarrow+\infty.
\end{aligned}
\end{eqnarray}
Combining (\ref{psi_xianshi}) and (\ref{tildepsi}) yields
\begin{eqnarray*}
\begin{aligned}
|\psi_n-\psi(t_n)|&\leq|\tilde{\psi}_n-\psi(t_n)|+(e^{C_b}-1)C_t\big(\frac{e^{C_bT}-1}{C_b^2}+\frac{C\Delta t}{C_b}\big)\Delta t\\
&\leq \Big(C+(e^{C_b}-1)C_t\big(\frac{e^{C_bT}-1}{C_b^2}+\frac{C\Delta t}{C_b}\big)\Big)\Delta t,\quad N\rightarrow+\infty.
\end{aligned}
\end{eqnarray*}
\end{proof}

Let $\varepsilon_N=\sup_i\|J^{\prime}(u^{i,N})-J_{N}^{\prime}(u^{i,N})\|$, where $J^{\prime}(\cdot)$ and $J_N^{\prime}(\cdot)$ are defined in (\ref{L_mu_dis}) and (\ref{fully_lj_mc}), then we have the following estimate for $\|u-u^{i+1,N}\|$.
\begin{Theorem}\label{u_uh}
Under the Assumptions \ref{jiashe1} and \ref{jiashe2}, let $u^{*}\in U(\delta)$ and $u^{i+1,N}$ be the solutions solved by SOCP (\ref{reduced})-(\ref{con_U}) and iterative scheme (\ref{fully_ita_lisan2}). We assume $u^{*},J^{\prime}(u^{*})$ are Lipschitz continuous functions and
$J^{\prime}(\cdot)$ is Lipschitz and uniformly monotone, i.e., there are positive constants $C_1,C_2$ such that
\begin{eqnarray*}
\begin{aligned}
\|J^{\prime}(u)-J^{\prime}(v)\|&\leq C_1\|u-v\|,\quad\forall u,v\in L^2([0,T];\mathbb{R}),\\
\big(J^{\prime}(u)-J^{\prime}(v),u-v\big)&\geq C_2\|u-v\|^2,\quad\forall u,v\in L^2([0,T];\mathbb{R}).
\end{aligned}
\end{eqnarray*}
Then when $\rho$ is chosen such that $0<1-2C_2\rho+(1+2C_1^2)\rho^2<1$, the following estimate holds:
\begin{eqnarray*}
\begin{aligned}
\|u^*-u^{i+1,N}\|&\leq C(\Delta t+\varepsilon_N),\quad i\rightarrow+\infty, \\
\end{aligned}
\end{eqnarray*}
where $C$ is a positive constant independent of $N$.
\end{Theorem}

\begin{proof}
By the triangle inequality, we have
\begin{eqnarray*}
\begin{aligned}
\|u^*-u^{i+1,N}\|&=\|u^*-u^{*,N}+u^{*,N}-u^{i+1,N}\|
\leq\|u^*-u^{*,N}\|+\|u^{*,N}-u^{i+1,N}\|.
\end{aligned}
\end{eqnarray*}
%Like (14) in \cite{max3}, employing Lemma \ref{xingzhi_P_dis}, the Lipschitz condition and monotonicity property of $J^{\prime}(\cdot)$ yields
%$$\|u^{*,N}-u^{i+1,N}\|^2\leq\left(1-2C_2\rho+(1+2C_1^2)\rho^2\right)\|u^{*,N}-u^{i,N}\|^2+(1+2\rho^2)\varepsilon_N^2.$$
From Corollary \ref{xingzhi_P_dis}, it holds that 
\begin{eqnarray*}
\begin{aligned}
\|u^{*,N}-u^{i+1,N}\|^2
&\leq\left\|u^{*,N}-u^{i,N}-\rho\left(J^{\prime}(u^{*,N})-J_{N}^{\prime}(u^{i,N})\right)\right\|^2\\
&=\|u^{*,N}-u^{i,N}\|^2-2\rho\left(u^{*,N}-u^{i,N}, J^{\prime}(u^{*,N})-J_{N}^{\prime}(u^{i,N})  \right)\\
&\quad+\rho^2\|J^{\prime}(u^{*,N})-J_{N}^{\prime}(u^{i,N}) \|^2.
\end{aligned}
\end{eqnarray*}
By the Lipschitz condition and the monotonicity property of $J^{\prime}(\cdot)$, we have
\begin{eqnarray*}
\begin{aligned}
&-2\rho\left(u^{*,N}-u^{i,N},J^{\prime}(u^{*,N})-J_N^{\prime}(u^{i,N})\right)\\
&=-2\rho\left(u^{*,N}-u^{i,N},J^{\prime}(u^{*,N})-J^{\prime}(u^{i,N})+J^{\prime}(u^{i,N})- J_N^{\prime}(u^{i,N})\right)\\
&\leq-2\rho C_2\|u^{*,N}-u^{i,N}\|^2+\rho^2\|u^{*,N}-u^{i,N}\|^2+\varepsilon_N^2
\end{aligned}
\end{eqnarray*}
and
\begin{eqnarray*}
\begin{aligned}
\rho^2\|J^{\prime}(u^{*,N})-J_N^{\prime}(u^{i,N})\|^2
&=\rho^2\|J^{\prime}(u^{*,N})-J^{\prime}(u^{i,N})+J^{\prime}(u^{i,N})-J_N^{\prime}(u^{i,N})\|^2\\
&\leq 2C_1^2\rho^2\|u^{*,N}-u^{i,N}\|^2+2\rho^2\varepsilon_N^2,
\end{aligned}
\end{eqnarray*}
which implies $$\|u^{*,N}-u^{i+1,N}\|^2\leq\left(1-2C_2\rho+(1+2C_1^2)\rho^2\right)\|u^{*,N}-u^{i,N}\|^2+(1+2\rho^2)\varepsilon_N^2.$$
It can be found that $1-2C_2\rho+(1+2C_1^2)\rho^2$ is a quadratic function with positive quadratic coefficient and axis of symmetry, and it is equal to $1$ at point $0$, which implies there exists a $\rho$ such that $0<1-2C_2\rho+(1+2C_1^2)\rho^2<1$.
Then let $\alpha=\sqrt{1-2C_2\rho+(1+2C_1^2)\rho^2}$, we have $0<\alpha<1$ and
\begin{eqnarray*}
\begin{aligned}
\|u^{*,N}-u^{i+1,N}\|\leq\alpha\|u^{*,N}-u^{i,N}\|+\sqrt{1+2\rho^2}\varepsilon_N
\leq\alpha^{i+1}\|u^{*,N}-u^{0,N}\|+\sqrt{1+2\rho^2}\frac{1-\alpha^{i+1}}{1-\alpha}\varepsilon_N.
\end{aligned}
\end{eqnarray*}
When $i\rightarrow+\infty$, we arrive at
\begin{eqnarray}\label{esta_1}
\begin{aligned}
\|u^{*,N}-u^{i+1,N}\|\leq\Delta t\|u^{*,N}-u^{0,N}\|+\frac{\sqrt{1+2\rho^2}}{1-\alpha}\varepsilon_N,\quad i\rightarrow+\infty.
\end{aligned}
\end{eqnarray}
For the estimate of $\|u^{*}-u^{*,N}\|$, we have 
\begin{eqnarray*}
\begin{aligned}
\|u^*-u^{*,N}\|&\leq\big\|u^*-\mathbb{P}_N\big(u^*-\rho J^{\prime}(u^*)\big)\big\|+\big\|\mathbb{P}_N\big(u^*-\rho J^{\prime}(u^*)\big)-u^{*,N}\big\|.
\end{aligned}
\end{eqnarray*}
By Corollary \ref{xingzhi_P_dis}, the Lipschitz condition and the monotonicity property of $J^{\prime}(\cdot)$, we arrive at
\begin{eqnarray}
\begin{aligned}\label{esta_3}
&\big\|\mathbb{P}_N\big(u^*-\rho J^{\prime}(u^*)\big)-u^{*,N}\big\|^2 \\
&\leq\big\|u^*-u^{*,N}-\rho\big(J^{\prime}(u^*)-
J^{\prime}(u^{*,N})\big)\big\|^2 \\
&=\|u^*-u^{*,N}\|^2-2\rho\big(u^*-u^{*,N},J^{\prime}(u^*)-J^{\prime}(u^{*,N})\big)
+\rho^2\|J^{\prime}(u^*)-J^{\prime}(u^{*,N})\|^2 \\
&\leq\|u^*-u^{*,N}\|^2-2C_2\rho\|u^*-u^{*,N}\|^2+C_1^2\rho^2\|u^*-u^{*,N}\|^2 \\
&\leq(1-2C_2\rho+C_1^2\rho^2)\|u^*-u^{*,N}\|^2 \\
&\leq \alpha^2\|u^*-u^{*,N}\|^2.
\end{aligned}
\end{eqnarray}
Then from the above estimate and Corollary \ref{xingzhi_P_con} we have
\begin{eqnarray}\label{esta_2}
\begin{aligned}
&\|u^*-u^{*,N}\|\\
&\leq\frac{1}{1-\alpha}\big\|u^*-\mathbb{P}_N\big(u^*-\rho J^{\prime}(u^*)\big)\big\|
=\frac{1}{1-\alpha}\big\|\mathbb{P}\big(u^*-\rho J^{\prime}(u^*)\big)-\mathbb{P}_N\big(u^*-\rho J^{\prime}(u^*)\big)\big\|\\
&\leq\frac{1}{1-\alpha}
\left(\big\|\mathbb{P}\big(u^*-\rho J^{\prime}(u^*)\big)-\mathbb{P}\big(\mathcal{P}\big(u^*-\rho J^{\prime}(u^*)\big)\big)\big\| \right.\\
&\quad\left.+\big\|\mathbb{P}\big(\mathcal{P}\big(u^*-\rho J^{\prime}(u^*)\big)\big)-\mathbb{P}_N\big(u^*-\rho J^{\prime}(u^*)\big)\big\|\right)\\
&\leq\frac{1}{1-\alpha}\left(\big\|u^*-\rho J^{\prime}(u^*)- \mathcal{P}\big(u^*-\rho J^{\prime}(u^*)\big)\big\|
+\big\|\mathbb{P}\big(\mathcal{P}\big(u^*-\rho J^{\prime}(u^*)\big)\big)-\mathbb{P}_N\big(u^*-\rho J^{\prime}(u^*)\big)\big\|\right).
\end{aligned}
\end{eqnarray}
For convenience, let $\hat{u}^{*}:=u^*-\rho J^{\prime}(u^*)$ and the definitions of $\mathbb{P}(\mathcal{P}\hat{u}^*)$ and $\mathbb{P}_N(\hat{u}^*)$ are given by
\begin{eqnarray*}
\begin{aligned}
&\mathbb{P}(\mathcal{P}\hat{u}^*):=\mathcal{P}\hat{u}^*-\rho\hat{\mu}^{*}\psi(t)b_{u}^{\prime}(t)
=\mathcal{P}\hat{u}^*-\rho\frac{\max\left\{\mathbb{E}\left[\int_{0}^{T}\hat{y}_t^{*}dt\right]-\delta,0 \right\}}{\rho\int_{0}^{T}\varphi(t)dt}\psi(t)b_{u}^{\prime}(t),\\
&\mathbb{P}_N(\hat{u}^*):=\mathcal{P}\hat{u}^*-\rho\tilde{\mu}^{*,N}\mathcal{P}\big(\psi(t)b_{u}^{\prime}(t)\big)
=\mathcal{P}\hat{u}^*-\rho \frac{\max\left\{\mathbb{E}\left[\int_{0}^{T}\hat{y}_t^{*}dt\right]-\delta,0 \right\}}{\rho\int_{0}^{T}\tilde{\varphi}(t)dt}\mathcal{P}\big(\psi(t)b_{u}^{\prime}(t)\big),
\end{aligned}
\end{eqnarray*}
where $\hat{y}_t^{*}$ is the state variable solved by $\mathcal{P}\hat{u}^*$ as the control variable and $\psi(t),\varphi(t),\tilde{\varphi}(t)$ are the solutions of the equations (\ref{psi_xian}), (\ref{varphi_con}) and (\ref{varphi_dis}), respectively. By the proofs similar to Theorems \ref{mu_exist} and \ref{muh_exist}, it can be proved that 
\begin{eqnarray*}
\begin{aligned}
\big(\mathcal{P}\hat{u}^{*}-\mathbb{P}(\mathcal{P}\hat{u}^*),v-\mathbb{P}(\mathcal{P}\hat{u}^*)\big)\leq0,\quad 
\big(\hat{u}^{*}-\mathbb{P}_N(\hat{u}^*),w-\mathbb{P}_N(\hat{u}^*)\big)\leq0,
\end{aligned}
\end{eqnarray*}
for any $v\in U(\delta)$ and $w\in U^{N}(\delta)$. 

For the deterministic ODE in (\ref{psi_xian}), one can show that
\begin{eqnarray}\label{bound_psi}
\begin{aligned}
\frac{1-e^{-C_b(T-t)}}{C_b}\leq\psi(t)=\int_{t}^{T}e^{\int_{t}^{s}b_y^{\prime}(x)dx}ds\leq \frac{e^{C_b(T-t)}-1}{C_b},\quad t\in[0,T].
\end{aligned}
\end{eqnarray}
Further, by (\ref{bound_psi}) and $|b_{u}^{\prime}(t)|\leq C_b$, we have the following estimate for $\|\mathbb{P}(\mathcal{P}\hat{u}^*)-\mathbb{P}_N(\hat{u}^{*})\|$:
\begin{eqnarray}\label{est_p_pn}
\begin{aligned}
&\|\mathbb{P}(\mathcal{P}\hat{u}^*)-\mathbb{P}_N(\hat{u}^{*})\| 
=\big\|\rho\tilde{u}^{*}\mathcal{P}\big(\psi(t)b_{u}^{\prime}(t)\big)-\rho\hat{\mu}^{*}\psi(t)b_{u}^{\prime}(t)\big\| \\
&\leq\rho\tilde{u}^{*}\big\|\mathcal{P}\big(\psi(t)b_{u}^{\prime}(t)\big)-\psi(t)b_{u}^{\prime}(t)\big\|+
(\rho\tilde{u}^{*}-\rho\hat{u}^{*})\|\psi(t)b_{u}^{\prime}(t)\|\\
&\leq\rho\tilde{u}^{*}\big\|\mathcal{P}\big(\psi(t)b_{u}^{\prime}(t)\big)-\psi(t)b_{u}^{\prime}(t)\big\|+
(e^{C_bT}-1)T\frac{\max\left\{\mathbb{E}\left[\int_{0}^{T}\hat{y}_t^{*}dt\right]-\delta,0 \right\}}{\int_{0}^{T}\tilde{\varphi}(t)dt\int_{0}^{T}\varphi(t)dt}\|\varphi(t)-\tilde{\varphi}(t)\|.
%&=\Big\|\frac{\max\left\{\mathbb{E}\left[\int_{0}^{T}\hat{y}_t^{*}dt\right]-\delta,0 \right\}}{\int_{0}^{T}\tilde{\varphi}(t)dt}\mathcal{P}\big(\psi(t)b_{u}^{'}(t)\big)-
%\frac{\max\left\{\mathbb{E}\left[\int_{0}^{T}\hat{y}_t^{*}dt\right]-\delta,0 \right\}}{\int_{0}^{T}\varphi(t)dt}\psi(t)b_{u}^{'}(t)\Big\| \nonumber\\
%&\leq\frac{\max\left\{\mathbb{E}\left[\int_{0}^{T}\hat{y}_t^{*}dt\right]-\delta,0 \right\}}{\int_{0}^{T}\tilde{\varphi}(t)dt}\left\|\mathcal{P}\left(\psi(t)b_{u}^{'}(t)\right)-\psi(t)b_{u}^{'}(t)\right\| \nonumber\\
%&\quad+\left\|\left(\frac{\max\left\{\mathbb{E}\left[\int_{0}^{T}\hat{y}_t^{*}dt\right]-\delta,0 \right\}\int_{0}^{T}\varphi(t)-\tilde{\varphi}(t)dt}{\int_{0}^{T}\tilde{\varphi}(t)dt\int_{0}^{T}\varphi(t)dt}\right)\psi(t)b_{u}^{'}(t)\right\| \nonumber\\
%&\leq\frac{\max\left\{\mathbb{E}\left[\int_{0}^{T}\hat{y}_t^{*}dt\right]-\delta,0 \right\}}{\int_{0}^{T}\tilde{\varphi}(t)dt}\left\|\mathcal{P}\left(\psi(t)b_{u}^{'}(t)\right)-\psi(t)b_{u}^{'}(t)\right\| \nonumber\\
%&\quad+(e^{C_bT}-1)T\frac{\max\left\{\mathbb{E}\left[\int_{0}^{T}\hat{y}_t^{*}dt\right]-\delta,0 \right\}}{\int_{0}^{T}\tilde{\varphi}(t)dt\int_{0}^{T}\varphi(t)dt}\|\varphi(t)-\tilde{\varphi}(t)\|.
\end{aligned}
\end{eqnarray}
Since (\ref{varphi_con}) and (\ref{varphi_dis}) are deterministic ODEs, which implies that we can solve them analytically and have the following estimate:
\begin{align*}
\|\varphi(t)-\tilde{\varphi}(t)\|
&=\left(\int_{0}^{T}\left(\int_{0}^{t}b_{u}^{\prime}(s)\big(b_{u}^{\prime}(s)\psi(s)-\mathcal{P}\big(b_{u}^{\prime}(s)\psi(s)\big)\big)e^{
\int_{s}^{t}b_{y}^{\prime}(x)dx}  ds\right)^2  dt\right)^{1/2}\\
&\leq\sqrt{T}\left(\int_{0}^{T} \int_{0}^{t}\big(b_{u}^{\prime}(s)\big)^2\big(b_{u}^{\prime}(s)\psi(s)-\mathcal{P}\big(b_{u}^{\prime}(s)\psi(s)\big)\big)^2e^{2\int_{s}^{t}b_{y}^{\prime}(x)dx}ds   dt\right)^{1/2}\\
&\leq\sqrt{T}C_be^{C_bT}\left(\int_{0}^{T} \int_{0}^{t}   \big(b_{u}^{\prime}(s)\psi(s)-\mathcal{P}\big(b_{u}^{\prime}(s)\psi(s)\big) \big)^2ds dt\right)^{1/2}\\
&\leq C_be^{C_bT}T\big\|b_{u}^{\prime}(t)\psi(t)-\mathcal{P}\big(b_{u}^{\prime}(t)\psi(t)\big)\big\|.
\end{align*}
Combining the above estimate and (\ref{est_p_pn}), (\ref{esta_2}) is further deduced as follows:
\begin{eqnarray}\label{u_uN}
\begin{aligned}
&\|u^*-u^{*,N}\|\\
&\leq\frac{1}{1-\alpha}\left(\|\hat{u}^{*}-\mathcal{P}\hat{u}^{*}\|
+\rho\tilde{u}^{*}\big\|b_{u}^{\prime}(t)\psi(t)-\mathcal{P}\big(b_{u}^{\prime}(t)\psi(t)
\big)\big\|\right.\\
%+\frac{\max\left\{\mathbb{E}\left[\int_{0}^{T}\hat{y}_t^{*}dt\right]-\delta,0 \right\}}{\int_{0}^{T}\tilde{\varphi}(t)dt}\left\|b_{u}^{'}(t)\psi(t)-\mathcal{P}\left(b_{u}^{'}(t)\psi(t)
%\right)\right\|\right.\\
&\quad+\left.\frac{\max\left\{\mathbb{E}\left[\int_{0}^{T}\hat{y}_t^{*}dt\right]-\delta,0 \right\}}{\int_{0}^{T}\tilde{\varphi}(t)dt\int_{0}^{T}\varphi(t)dt}C_b(e^{C_bT}-1)e^{C_bT}T^2
\big\|b_{u}^{\prime}(t)\psi(t)-\mathcal{P}\big(b_{u}^{\prime}(t)\psi(t)\big)\big\|\right).
\end{aligned}
\end{eqnarray}
Employing the Lipschitz continuity of $u^*-\rho J^{\prime}(u^*)$ and the mean value theorem yields
\begin{eqnarray}\label{Lipschitz1}
\begin{aligned}
&\|\hat{u}^{*}-\mathcal{P}\hat{u}^{*}\|=\left(\sum_{n=0}^{N-1}\int_{t_n}^{t_{n+1}}\left(\hat{u}^{*}-\frac{1}{\Delta t}\int_{t_n}^{t_{n+1}}\hat{u}^{*}dt  \right)^2dt\right)^{1/2}\\
%\left(\sum_{n=0}^{N-1}\int_{t_n}^{t_{n+1}}\left(\hat{u}^{*}-\mathcal{P}\hat{u}^{*}\right)^2dt\right)^{1/2}
%\\
%&=\left(\sum_{n=0}^{N-1}\int_{t_n}^{t_{n+1}}\left(\hat{u}^{*}-\frac{1}{\Delta t}\int_{t_n}^{t_{n+1}}\hat{u}^{*}dt  \right)^2dt\right)^{1/2}
&=\left(\sum_{n=0}^{N-1}\int_{t_n}^{t_{n+1}}\left(\hat{u}^{*}-\hat{u}^{*}(\xi_n)   \right)^2dt\right)^{1/2}
\leq C\Delta t,
\end{aligned}
\end{eqnarray}
where $\xi_n\in[t_n,t_{n+1}]$.

From (\ref{psi_xian}) and (\ref{bound_psi}), we have the following estimate for $\psi^{\prime}(t)$:
\begin{eqnarray}\label{daoshu_psi}
\begin{aligned}
%\frac{1-e^{-C_b(T-t)}}{C_b}&\leq\psi(t)=\int_{t}^{T}e^{\int_{t}^{s}b_y^{'}(x)dx}ds\leq \frac{e^{C_b(T-t)}-1}{C_b},\quad t\in[0,T],\\
-e^{C_b(T-t)}\leq-C_b\psi(t)-1\leq\psi^{\prime}(t)=-1-\psi(t)b_y^{\prime}(t)\leq C_b\psi(t)-1\leq e^{C_b(T-t)}-2,
%&\psi^{'}(t)=-1-\psi(t)b_y^{'}(t)\leq C_b\psi(t)-1\leq e^{C_b(T-t)}-2,
\end{aligned}
\end{eqnarray}
which implies that $|\psi^{\prime}(t)|\leq e^{C_bT}$ and $b_{u}^{\prime}(t)\psi(t)$ is Lipschitz continuity. Like (\ref{Lipschitz1}), it holds that
\begin{eqnarray*}
\begin{aligned}
\big\|b_{u}^{\prime}(t)\psi(t)-\mathcal{P}\big(b_{u}^{\prime}(t)\psi(t)\big)\big\|\leq C\Delta t,
\end{aligned}
\end{eqnarray*}
which implies based on (\ref{u_uN})
\begin{eqnarray}\label{lipu_esti}
\begin{aligned}
\|u^*-u^{*,N}\|\leq C\Delta t.
\end{aligned}
\end{eqnarray}
Combining (\ref{esta_1}) and (\ref{lipu_esti}) yields the final result.
\end{proof}

Based on the error estimate of control, we have the following estimate for the multiplier.
\begin{Theorem}\label{mu_muh}
Under the assumptions in Theorem \ref{u_uh}, for sufficiently small $\Delta t$, the following estimate holds
\begin{eqnarray*}
\begin{aligned}
|\mu^*-\mu_{h}^{i,N}|&\leq C(\varepsilon_N+\Delta t),\quad i\rightarrow+\infty,
\end{aligned}
\end{eqnarray*}
where $C$ is a positive constant independent of $N$.
\end{Theorem}
\begin{proof}
Using Corollary \ref{xingzhi_P_con} and the similar arguments as for deriving (\ref{esta_3}) yields
\begin{eqnarray*}
\begin{aligned}
\|u^{*}-u^{i+1,*}\|^2\leq\big\|u^{*}-u^{i,*}-\rho\big(J^{\prime}(u^*)-J^{\prime}(u^{i,*})\big) \big\|^2\leq\alpha^2\|u^{*}-u^{i,*}\|^2,
\end{aligned}
\end{eqnarray*}
which implies 
\begin{eqnarray}\label{fuzhu_mu2}
\begin{aligned}
\|u^*-u^{i+1,*}\|\leq\alpha^{i+1}\|u^*-u^{0,*}\|\leq\Delta t\|u^*-u^{0,*}\|,\quad i\rightarrow+\infty.
\end{aligned}
\end{eqnarray}
Employing the triangle inequality, the above estimate and Theorem \ref{u_uh} yields
\begin{eqnarray}\label{fuzhu_mu1}
\begin{aligned}
\|u^{i+1,*}-u^{i+1,N}\|
\leq\|u^*-u^{i+1,*}\|+\|u^{*}-u^{i+1,N}\|\leq C(\Delta t+\varepsilon_N),\quad i\rightarrow+\infty.
\end{aligned}
\end{eqnarray}

Using the triangle inequality again yields
\begin{eqnarray*}
\begin{aligned}
|\mu^*-\mu_{h}^{i,N}|=|\mu^*-\mu^{i,*}|+|\mu^{i,*}-\mu_{h}^{i,N}|.
\end{aligned}
\end{eqnarray*}
By the fact that $\mathcal{L}_{u}^{\prime}(u^{*},\mu^{*})=0$ in (\ref{grad_lage}) as well as (\ref{ita_con}), (\ref{pro_xian}), we arrive at
\begin{eqnarray}\label{mui}
\begin{aligned}
&\mathcal{L}_{u}^{\prime}(u^{*},\mu^{*})=J^{\prime}(u^{*})+\mu^{*}\psi(t)b_{u}^{\prime}(t)=0,\\
&u^{i+1,*}=\mathbb{P}\big(u^{i,*}-\rho J^{\prime}(u^{i,*})\big)=u^{i,*}-\rho J^{\prime}(u^{i,*})-\rho\mu^{i,*}\psi(t)b_{u}^{\prime}(t).
\end{aligned}
\end{eqnarray}
Differentiating the above two equations yields
\begin{eqnarray*}
\begin{aligned}
\rho\|J^{\prime}(u^*)-J^{\prime}(u^{i,*})+(\mu^{*}-\mu^{i,*})\psi(t)b_{u}^{\prime}(t)\|
\leq\|u^{i+1,*}-u^{*}\|+\|u^{i,*}-u^{*}\|.
\end{aligned}
\end{eqnarray*}
Employing the Lipschitz property of $J^{\prime}(\cdot)$ yields
\begin{eqnarray}\label{mu_mu_i}
\begin{aligned}
|\mu^{*}-\mu^{i,*}|\|\psi(t)b_{u}^{\prime}(t)\|\leq(C_1+\frac{1}{\rho})\|u^{*}-u^{i,*}\|+\frac{1}{\rho}\|u^{*}-u^{i+1,*}\|.
\end{aligned}
\end{eqnarray}
%where $\psi(t)=\int_{t}^{T}e^{\int_{t}^{s}b_y^{'}(y_t^{i,*},u^{i,*})|_{t=x}dx}ds$.
Then by the Cauchy-Schwarz inequality and (\ref{bound_psi}), we arrive at
\begin{eqnarray*}
\begin{aligned}
\|\psi(t)b_{u}^{\prime}(t)\|
%&\geq c_b\|\psi(t)\|
\geq \frac{c_b\sqrt{T}}{T}\int_{0}^{T}\psi(t)dt
=\frac{c_b\sqrt{T}}{T}\int_{0}^{T}\int_{t}^{T}e^{\int_{t}^{s}b_y^{\prime}(x)dx}dsdt
\geq\frac{c_b\sqrt{T}(C_bT+e^{-C_bT}-1)}{TC_b^2}.
\end{aligned}
\end{eqnarray*}
According to estimate (\ref{fuzhu_mu2}), $|\mu^{*}-\mu^{i,*}|$ can be bounded as follows:
\begin{eqnarray}\label{estamu_1}
\begin{aligned}
|\mu^{*}-\mu^{i,*}|
&\leq \frac{\sqrt{T}C_b^2(C_1\rho+1)}{\rho c_b(C_bT+e^{-C_bT}-1)}\|u^{*}-u^{i,*}\|
+\frac{\sqrt{T}C_b^2}{\rho c_b(C_bT+e^{-C_bT}-1)}\|u^{*}-u^{i+1,*}\|\\
&\leq\frac{\sqrt{T}C_b^2(C_1\rho+1)}{\rho c_b(C_bT+e^{-C_bT}-1)}\alpha^{i}\|u^{*}-u^{0,*}\|\\
&\quad+\frac{\sqrt{T}C_b^2}{\rho c_b(C_bT+e^{-C_bT}-1)}\alpha^{i+1}\|u^{*}-u^{0,*}\|\\
&\leq C\Delta t,\quad i\rightarrow+\infty.
\end{aligned}
\end{eqnarray}

Next the error $|\mu^{i,*}-\mu_{h}^{i,N}|$ will be estimated, where $\mu^{i,*}$ and $\mu_{h}^{i,N}$ satisfy   the second equation of (\ref{mui}) and the following equation, respectively,
\begin{eqnarray*}
\begin{aligned}
u^{i+1,N}&
%=u^{i,N}-\rho\mathcal{L}_{u,N}^{'}
=u^{i,N}-\rho 	J_N^{\prime}(u^{i,N})-\rho\mu_{h}^{i,N}
\sum_{n=0}^{N-1}b_{u}^{\prime}(t_n)\psi_{n}\chi_{I_{n}^{N}}(t).
\end{aligned}
\end{eqnarray*}
After a similar derivation to (\ref{mu_mu_i}), we get
\begin{eqnarray*}
\begin{aligned}
\big\|\mu_{h}^{i,N}
\sum_{n=0}^{N-1}\psi_{n}b_{u}^{\prime}(t_n)\chi_{I_{n}^{N}}(t)-\mu^{i,*}\psi(t)b_{u}^{\prime}(t)\big\|
\leq \frac{C_1\rho+1}{\rho}\|u^{i,*}-u^{i,N}\|+\frac{1}{\rho}\|u^{i+1,*}-u^{i+1,N}\|+\varepsilon_N.
\end{aligned}
\end{eqnarray*}
%Due the linear property of $b$, we know $b_{u}^{'}(y_{t_n},u^{i,N}(t_n))=b_{u}^{'}(t_n)$ and $b_{u}^{'}(y_t^{i,*},u^{i,*})=b_{u}^{'}(t)$.
%Further, we denote by $\psi^{\pi}(t)$ the piecewise constant interpolation for $\psi_n$, i.e., $\psi^{\pi}(t)=\sum_{n=0}^{N-1}\psi_{n}\chi_{I_{n}^{N}}(t)$. 
By triangle inequality and (\ref{bound_psi}) it holds that
\begin{eqnarray}
\begin{aligned}\label{mu_mu_N}
&|\mu_{h}^{i,N}-\mu^{i,*}|\big\|\sum_{n=0}^{N-1}\psi_{n}b_{u}^{\prime}(t_n)\chi_{I_{n}^{N}}(t)\big\| \\
&\leq\big\|\mu^{i,*}\left(\psi^{\pi}(t)-\psi(t)\right)\sum_{n=0}^{N-1}b_{u}^{\prime}(t_n)\chi_{I_{n}^{N}}(t)
\big\|
+\big\|\mu^{i,*}\psi(t)\big(\sum_{n=0}^{N-1}b_{u}^{\prime}(t_n)\chi_{I_{n}^{N}}(t)-b_{u}^{\prime}(t)\big)
\big\| \\
&\quad+\frac{C_1\rho+1}{\rho}\|u^{i,*}-u^{i,N}\|+\frac{1}{\rho}\|u^{i+1,*}-u^{i+1,N}\|+\varepsilon_N \\
&\leq|\mu^{i,*}|C_b\|\psi^{\pi}(t)-\psi(t)\|
+|\mu^{i,*}|\frac{e^{C_bT}-1}{C_b}\big\|\sum_{n=0}^{N-1}b_{u}^{\prime}(t_n)\chi_{I_{n}^{N}}(t)-b_{u}^{\prime}(t)
\big\| \\
&\quad+\frac{C_1\rho+1}{\rho}\|u^{i,*}-u^{i,N}\|+\frac{1}{\rho}\|u^{i+1,*}-u^{i+1,N}\|+\varepsilon_N.
\end{aligned}
\end{eqnarray}
%By the standard error estimate of implicit Euler discretization in (\ref{psi_chazhi}) for (\ref{psi_xian}), for sufficiently small $\Delta t$, we have
%\begin{eqnarray*}\label{estimate_psi}
%\begin{aligned}
%\max_{0\leq n\leq N-1}|\psi_n-\psi(t_n)|\leq C_{\psi}\Delta t.
%\end{aligned}
%\end{eqnarray*}
Based on Theorem \ref{psi_theorem}, (\ref{daoshu_psi}) and the Taylor expansion, we can bound $\|\psi^{\pi}(t)-\psi(t)\|$ as follows:
\begin{align}\label{esti_mu}
&\|\psi^{\pi}(t)-\psi(t)\| 
=\left(\sum\limits_{n=0}^{N-1}\int_{t_n}^{t_{n+1}}\big(\psi_n-\psi(t)\big)^2dt\right)^{1/2}\nonumber\\ 
&=\left(\sum\limits_{n=0}^{N-1}\int_{t_n}^{t_{n+1}}\big(\psi(t_n)+\psi^{\prime}\left(t_n+\theta(t-t_n)\right)(t-t_n)-\psi_n\big)^2dt\right)^{1/2} \nonumber\\ 
&\leq\left(\sum\limits_{n=0}^{N-1}\left( \int_{t_n}^{t_{n+1}}\big(\psi(t_n)-\psi_n\big)^2dt+ \int_{t_n}^{t_{n+1}} \big(\psi^{\prime}\left(t_n+\theta(t-t_n)\right)(t-t_n)\big)^2dt \right.\right. \nonumber\\ 
&\quad\left.\left.+\int_{t_n}^{t_{n+1}}2|\psi(t_n)-\psi_n||\psi^{\prime}\left(t_n+\theta(t-t_n)\right)| (t-t_n)dt \right)\right)^{1/2} \nonumber\\ 
&\leq\left(\sum\limits_{n=0}^{N-1}\left( \big(\psi(t_n)-\psi_n\big)^2\Delta t+ e^{2C_bT}(\Delta t)^3
+2e^{C_bT}|\psi(t_n)-\psi_n|(\Delta t)^2 \right)\right)^{1/2} \nonumber\\ 
&\leq \left(\max_{0\leq n\leq N-1}|\psi(t_n)-\psi_n|^2+e^{2C_bT}(\Delta t)^2 
+2e^{C_bT}\max_{0\leq n\leq N-1}|\psi(t_n)-\psi_n|\Delta t    \right)^{1/2} \nonumber\\
&\leq \sqrt{C_{\psi}^2+e^{2C_bT}+2e^{C_bT}C_{\psi}}\Delta t,
\end{align}
where $\theta\in(0,1)$. Similarly, we have
\begin{eqnarray}\label{esti_mu1}
\begin{aligned}
&\big\|\sum_{n=0}^{N-1}b_{u}^{\prime}(t_n)\chi_{I_{n}^{N}}(t)-b_{u}^{\prime}(t)\big\|
=\left(\sum_{n=0}^{N-1}\int_{t_n}^{t_{n+1}}\big( b_{u}^{\prime}(t_n) - b_{u}^{\prime}(t) \big)^2dt\right)^{1/2}\\
&=\left(\sum_{n=0}^{N-1}\int_{t_n}^{t_{n+1}}\big( b_{u,t}^{\prime\prime}\left(t_n+\theta(t-t_n)\right)(t-t_n) \big)^2dt\right)^{1/2}\leq C_t\Delta t.
\end{aligned}
\end{eqnarray}
Meanwhile, from the estimates of (\ref{esti_mu}) and $\|\psi\|$ we get
\begin{align*}
&|\mu_{h}^{i,N}-\mu^{i,*}|\big\|\psi^{\pi}(t)\sum_{n=0}^{N-1}b_{u}^{\prime}(t_n)\chi_{I_{n}^{N}}(t)\big\|
\geq |\mu_{h}^{i,N}-\mu^{i,*}|c_b\|\psi^{\pi}(t)\|\\
&\geq |\mu_{h}^{i,N}-\mu^{i,*}|c_b\left(\|\psi(t)\|-\sqrt{C_{\psi}^2+e^{2C_bT}+2e^{C_bT}C_{\psi}}\Delta t\right)\\
&\geq|\mu_{h,fu}^{i,N}-\mu^{i,*}|c_b \left(\frac{\sqrt{T}(C_bT+e^{-C_bT}-1)}{TC_b^2}-
\sqrt{C_{\psi}^2+e^{2C_bT}+2e^{C_bT}C_{\psi}}\Delta t\right).
\end{align*}
For sufficiently small $\Delta t$ such that
$\Delta t\leq\frac{\sqrt{T}(C_bT+e^{-C_bT}-1)}{2TC_b^2\sqrt{C_{\psi}^2+e^{2C_bT}+2e^{C_bT}C_{\psi}}}$, which implies
\begin{eqnarray*}
\begin{aligned}
|\mu_{h}^{i,N}-\mu^{i,*}|\big\|\psi^{\pi}(t)\sum_{n=0}^{N-1}b_{u}^{\prime}(t_n)\chi_{I_{n}^{N}}(t)\big\|
\geq |\mu_{h}^{i,N}-\mu^{i,*}|\frac{c_b\sqrt{T}(C_bT+e^{-C_bT}-1)}{2TC_b^2}.
\end{aligned}
\end{eqnarray*}
Taking the above estimates into (\ref{mu_mu_N}) and using (\ref{esti_mu}), (\ref{esti_mu1}) as well as (\ref{fuzhu_mu1}) yields 
\begin{align}\label{estamu_2}
&|\mu_{h}^{i,N}-\mu^{i,*}|\nonumber\\
&\leq\frac{2\sqrt{T}C_b^2}{c_b(C_bT+e^{-C_bT}-1)}
\left(\frac{C_1\rho+1}{\rho}\|u^{i,*}-u^{i,N}\|+\frac{1}{\rho}\|u^{i+1,*}-u^{i+1,N}\|+\varepsilon_N  \right.\nonumber\\
&\quad\left.+|\mu^{i,*}|C_b\sqrt{C_{\psi}^2+e^{2C_bT}+2e^{C_bT}C_{\psi}}\Delta t+|\mu^{i,*}|\frac{e^{C_bT}-1}{C_b}C_t\Delta t \right)  \\
&\leq\frac{2\sqrt{T}C_b^2}{c_b(C_bT+e^{-C_bT}-1)}\left( \left(1+\frac{(C_1\rho+2)C}{\rho}\right)\varepsilon_N
+\frac{(C_1\rho+2)C}{\rho}\Delta t  \right.\nonumber\\
&\quad\left.+(\mu^{*}+C\Delta t)C_b\sqrt{C_{\psi}^2+e^{2C_bT}+2e^{C_bT}C_{\psi}}\Delta t
+(\mu^{*}+C\Delta t)\frac{C_t(e^{C_bT}-1)}{C_b}\Delta t\right),\ 
i\rightarrow+\infty.\nonumber
\end{align}
Combing (\ref{estamu_1}) and (\ref{estamu_2}) yields the result.
\end{proof}

Next, we aim to estimate $\varepsilon_N=\sup_{i}\|J^{\prime}(u^{i,N})-J^{\prime}_N(u^{i,N})\|$. 
%where $\hat{J}^{'}_N(\cdot)$ is defined in (\ref{fully_lj_mc}). 
Prior to this, some ancillary results are introduced.
According to \cite{other2}, we define a stochastic process $\tilde{y}_{t}^{\pi}$ as follows for $t\in[t_{n},t_{n+1}]$:
\begin{equation*}
\begin{aligned}\label{taolun1}
\tilde{y}_t^{\pi}=\tilde{y}_{n}^{\pi}+b\left(\tilde{y}_{n}^{\pi},u(t_n)\right)(t-t_n)+\sigma\left(\tilde{y}_{n}^{\pi},u(t_n)\right)(W_t-W_{n}),\ \tilde{y}_{0}^{\pi}=y^0.
\end{aligned}
\end{equation*}
Then, for sufficiently small $\Delta t$ and $j\geq 1$, we have the following strong and weak estimates:
\begin{equation}\label{estimate_sde}
\begin{aligned}
&\max_{0\leq n\leq N-1}\big(\sup_{0\leq t\leq T}\mathbb{E}\big[|y_t-\tilde{y}_t^{\pi}|^j\big]
+\sup_{t_n\leq t\leq t_{n+1}}\mathbb{E}\big[|y_t-y_{t_n}|^j\big] \big)^{1/j}=\mathcal{O}\big((\Delta t)^{1/2}\big),\\
&\max_{0\leq n\leq N-1}\big(\sup_{0\leq t\leq T}|\mathbb{E}[y_t-\tilde{y}_t^{\pi}]|
+\sup_{t_n\leq t\leq t_{n+1}}|\mathbb{E}[y_t-y_{t_n}]|\big)=\mathcal{O}(\Delta t).
%|\mathbb{E}(y_n-y_{t_n})|=\mathcal{O}(\Delta t),\ |\mathbb{E}(y_{t}-y_{t_n})|=\mathcal{O}(\Delta t).
\end{aligned}
\end{equation}
It is obvious that $y_n=\tilde{y}_{n}^{\pi}$ for the time node $t_n\in\Pi$. From (\ref{estimate_sde}), we further have the following results for $t\in [t_n,t_{n+1}]$:
\begin{equation}\label{estimate_sde1}
\begin{aligned}
&\big(\mathbb{E}\big[|y_t-y_n|^{j}\big]\big)^{1/j}\leq\big(
\mathbb{E}\big[|y_t-y_{t_n}|^{j}\big]\big)^{1/j}+\big(\mathbb{E}\big[|y_{t_n}-\tilde{y}_{n}^{\pi}|^{j}\big]\big)^{1/j}=\mathcal{O}\big((\Delta t)^{1/2}\big),\\
&|\mathbb{E}[y_t-y_n]|\leq |\mathbb{E}[y_t-y_{t_n}]|+|\mathbb{E}[y_{t_n}-\tilde{y}_{n}^{\pi}]|=
\mathcal{O}(\Delta t).
\end{aligned}
\end{equation}
According to Proposition 4, Theorem 7 and Theorem 8 in \cite{other3}, we have the following lemma.
\begin{Lemma}(\cite{other3})\label{fuzhu_pq}
If the functions $b$, $\sigma$, $\hat{f}$ and $g$ are bounded in $y$, are uniformly Lipschitz continuous w.r.t. $(y,p,q)$ and H\"{o}lder continuous of parameter $\frac{1}{2}$ w.r.t. $t$. In addition, $b,\sigma\in C_{b}^{4,2},\hat{f}\in C_{b}^{4,4,4,2},g\in C_{b}^{4+\alpha}$, then it holds that 
$\hat{\eta}(t,y)\in C_{b}^{2,4}$ and
\begin{equation*}\label{xianyan}
\begin{aligned}
%&\max_{0\leq i\leq N}\mathbb{E}[(y_{t_i}^\pi-y_{t_i})^2]=\mathcal{O}(\Delta t),\\
\hat{p}_{n}-\hat{p}_{t_n}&=\hat{\eta}_y^{\prime}(t_n,y_{t_n})(\tilde{y}_{n}^{\pi}-y_{t_n})+\mathcal{O}(\Delta t)
+\mathcal{O}\big(|\tilde{y}_{n}^{\pi}-y_{t_n}|^2\big),\\
\hat{q}_{n}-\hat{q}_{t_n}&=[\hat{\eta}_y^{\prime}\sigma]_{y}^{\prime}(t_n,y_{t_n})(\tilde{y}_{n}^{\pi}-y_{t_n})+\mathcal{O}(\Delta t)+\mathcal{O}\big(|\tilde{y}_{n}^{\pi}-y_{t_n}|^2\big).
\end{aligned}
\end{equation*}
\end{Lemma}
The above lemma implies the following results by combining (\ref{estimate_sde}):
\begin{equation}\label{weak_pq}
\begin{aligned}
|\mathbb{E}[\hat{p}_{n}-\hat{p}_{t_n}]|=\mathcal{O}(\Delta t),\quad 
|\mathbb{E}[\hat{q}_{n}-\hat{q}_{t_n}]|=\mathcal{O}(\Delta t).
\end{aligned}
\end{equation}

From the LSMC discretization processes in subsection \ref{lsmc_process}, it is clear that the estimate between $J^{\prime}(u)$ and $J^{\prime}_{N}(u)$ is determined by projection error produced in (\ref{infinite}) and Monte Carlo simulation error produced in (\ref{MC1}). For the sake of clarity, we denote $\big(\hat{P}_{n}^{K}(y_{n}),\hat{Q}_{n}^{\tilde{K}}(y_{n})\big)$ as follows:
\begin{equation*}
\left\{\begin{aligned}
\hat{Q}_{n}^{\tilde{K}}(y_{n})&=\hat{\mathcal{P}}_{1,n}\big(\frac{\Delta W_{n+1}}{\Delta t}\hat{P}_{n+1}^{K}(y_{n+1})\big),\\
\hat{P}_{n}^{K}(y_{n})&=\hat{\mathcal{P}}_{0,n}\big(\hat{P}_{n+1}^{K}(y_{n+1})
+\hat{f}\big(y_{n},\hat{P}_{n+1}^{K}(y_{n+1}),\hat{Q}_{n}^{\tilde{K}}(y_{n}),u(t_n)\big)\Delta t\big).
\end{aligned}\right.
\end{equation*}
The definitions of $\hat{\mathcal{P}}_{1,n}(\cdot)$ and $\hat{\mathcal{P}}_{0,n}(\cdot)$ can be obtained by (\ref{infinite}). 
%It follows that $\hat{\mathcal{P}}_{1,n}(\cdot)$ and $\hat{\mathcal{P}}_{0,n}(\cdot)$ are actually the orthogonal projections onto spaces $\Lambda_{0,n}$ and $\Lambda_{1,n}$, and 
%They are used to approximate the conditional expectations in (\ref{disscheme_fu}). 
From \cite{other4}, we have the following estimate for the projection error.
\begin{Lemma}(\cite{other4})\label{error_pro}
Under the assumption that $\hat{f}$ is Lipschitz continuous with respect to $p,q$, the following projection error estimate holds:
\begin{equation}\label{pkqk}
\begin{aligned}
&\max_{0\leq n\leq N}\mathbb{E}\big[\big(\hat{p}_{n}-\hat{P}_{n}^{K}(y_n)\big)^2\big]+\sum\limits_{n=0}^{N-1}\Delta t\mathbb{E}\big[\big(\hat{q}_{n}-\hat{Q}_{n}^{\tilde{K}}(y_n)\big)^2\big]\\
&\quad\leq C\left(\sum\limits_{n=0}^{N-1}\mathbb{E}\big[\big(\hat{\mathcal{P}}_{0,n}\hat{p}_{n}-\hat{p}_{n}\big)^2\big]
+\sum\limits_{n=0}^{N-1}\Delta t\mathbb{E}\big[\big(\hat{\mathcal{P}}_{1,n}\hat{q}_{n}-\hat{q}_{n}\big)^2\big]\right),
\end{aligned}
\end{equation}
where $C$ is a positive constant independent of $N$.
\end{Lemma}
\begin{Remark}
The explicit form of the above projection error depends on the choice of the basis functions and it is difficult to quantify except for some special classes of basis functions such as indicator functions of hypercubes. Following \cite{goblt1} and \cite{goblt2}, it follows that $\mathbb{E}\big[\big(\hat{\mathcal{P}}_{0,n}\hat{p}_{n}-\hat{p}_{n}\big)^2\big]$ and $\mathbb{E}\big[\Delta t\big(\hat{\mathcal{P}}_{1,n}\hat{q}_{n}-\hat{q}_{n}\big)^2\big]$ can be bounded by $C\tau^2$ for all $n=0,1,...,N-1$, where $\tau$ denotes the edge length of the hypercube.
%Then, the dimension of the function bases $K$ and $\tilde{K}$ grows proportional to $\delta^{-1}$.
Therefore, if we choose $\tau=\mathcal{O}\left((\Delta t)^{3/2}\right)$, estimate (\ref{pkqk}) reaches $\mathcal{O}\left((\Delta t)^2\right)$.
%the following explicit estimate is obtained:
%\begin{equation}\label{project_erjie}
%\begin{aligned}
%\max_{0\leq n\leq N}\mathbb{E}\big[\big(\hat{p}_{n}-\hat{P}_{n}^{K}(y_n)\big)^2\big]+\sum\limits_{n=0}^{N-1}\Delta t\mathbb{E}\big[\big(\hat{q}_{n}-\hat{Q}_{n}^{\tilde{K}}(y_n)\big)^2\big]=\mathcal{O}\big((\Delta t)^2\big).
%\end{aligned}
%\end{equation}
\end{Remark}
We define $\mathbb{J}_{N}^{\prime}(u)$ as follows:
\begin{equation*}
\begin{aligned}
\mathbb{J}_{N}^{\prime}(u):=\mathbb{E}\big[\hat{P}_t^{K,L,\pi}
\sum_{n=0}^{N-1}b_{u}^{\prime}(t_n)\chi_{I_{n}^{N}}(t)
+\hat{Q}_{t}^{\tilde{K},L,\pi}\sigma_{u}^{\prime}(y_{t}^{\pi},u)\big]+
j^{\prime}(u).
\end{aligned}
\end{equation*}
For the $u\in U_{N}$, it is obvious that $J_{N}^{\prime}(u)$ defined in (\ref{fully_lj_mc}) is the Monte Carlo approximation of $\mathbb{J}_{N}^{\prime}(u)$ with $L$ simulation path. Therefore, using the Monte Carlo estimate yields
\begin{equation}
\begin{aligned}\label{moni_mc}
&\|J_{N}^{\prime}(u)-\mathbb{J}_{N}^{\prime}(u)\|\\
&=\left(\sum\limits_{n=0}^{N-1}\int_{t_n}^{t_{n+1}}\left(\mathbb{E}\big[\hat{P}_t^{K,L,\pi}
\sum_{n=0}^{N-1}b_{u}^{\prime}(t_n)\chi_{I_{n}^{N}}(t)
+\hat{Q}_{t}^{\tilde{K},L,\pi}\sigma_{u}^{\prime}(y_{t}^{\pi},u)\big] \right.\right.\\
&\quad\left.\left.-\frac{1}{L}\sum_{\lambda=1}^{L}\big[_{\lambda}\hat{P}_t^{K,L,\pi}
\sum_{n=0}^{N-1}b_{u}^{\prime}(t_n)\chi_{I_{n}^{N}}(t)
+_{\lambda}\hat{Q}_{t}^{\tilde{K},L,\pi}\sigma_{u}^{\prime}(_{\lambda}y_{t}^{\pi},u)\big]
\right)^2dt\right)^{1/2}\\
&=\left(\sum\limits_{n=0}^{N-1}\left(\mathbb{E}\big[\hat{P}_n^{K,L}(y_n)
b_{u}^{\prime}(t_n)+\hat{Q}_{n}^{\tilde{K},L}(y_n)\sigma_{u}^{\prime}\left(y_{n},u(t_n)\right)\big] \right.\right.\\
&\quad\left.\left.-\frac{1}{L}\sum_{\lambda=1}^{L}\big[\hat{P}_n^{K,L}(_{\lambda}y_n)
b_{u}^{\prime}(t_n)+\hat{Q}_{n}^{\tilde{K},L}(_{\lambda}y_n)\sigma_{u}^{\prime}\left(_{\lambda}y_{n},u(t_n)\right)
\big]\right)^2\Delta t \right)^{1/2}
%&\leq\max_{0\leq n\leq N-1}|\mathbb{E}[\hat{P}_n^{K,L}(y_n)
%b_{u}^{'}(t_n)+\hat{Q}_{n}^{\tilde{K},L}(y_n)\sigma_{u}^{'}(y_{n},u(t_n))]\\
%&\quad-\frac{1}{L}\sum_{\lambda=1}^{L}[\hat{P}_n^{K,L}(_{\lambda}y_n)
%b_{u}^{'}(t_n)+\hat{Q}_{n}^{\tilde{K},L}(_{\lambda}y_n)\sigma_{u}^{'}(_{\lambda}y_{n},u(t_n))]|
=\mathcal{O}(L^{-1/2}).
\end{aligned}
\end{equation}
Then the following estimate for $\varepsilon_N$ is obtained.
\begin{Theorem}\label{JN}
Under the Assumptions \ref{jiashe1}, \ref{jiashe2} and the assumptions in Lemmas \ref{fuzhu_pq}, \ref{error_pro}, for sufficiently small $\Delta t$ and Monte Carlo simulation path $L\geq\mathcal{O}(N^2)$, the following estimate holds:
\begin{equation*}
\begin{aligned}
&\sup\limits_i\|J^{\prime}(u^{i,N})-J^{\prime}_N(u^{i,N})\|^2\\
&\leq C\left( (\Delta t)^2+ \sum\limits_{n=0}^{N-1}\mathbb{E}\big[\big(\hat{\mathcal{P}}_{0,n}\hat{p}_{n}-\hat{p}_{n}\big)^2\big]+\sum\limits_{n=0}^{N-1}\Delta t\mathbb{E}\big[\big(\hat{\mathcal{P}}_{1,n}\hat{q}_{n}-\hat{q}_{n}\big)^2\big] \right.\\
&\quad\left.+\max_{0\leq n\leq N}\mathbb{E}\big[\big(\hat{P}_{n}^{K}(y_n)-\hat{P}_{n}^{K,L}(y_n)\big)^2\big]+\sum\limits_{n=0}^{N-1}\Delta t\mathbb{E}\big[\big(\hat{Q}_{n}^{\tilde{K}}(y_n)-\hat{Q}_{n}^{\tilde{K},L}(y_n)\big)^2\big] \right),
\end{aligned}
\end{equation*}
where $C$ is a positive constant independent of $N$.
\end{Theorem}
\begin{proof}
For any $u\in U_{N}$ and $L\geq\mathcal{O}(N^2)$, using the triangle inequality and (\ref{moni_mc}) yields
\begin{equation*}
\begin{aligned}
\|J^{\prime}(u)-J^{\prime}_N(u)\|
&\leq \|J^{\prime}(u)-\mathbb{J}^{\prime}_{N}(u)\|+\|\mathbb{J}^{\prime}_{N}(u)-J^{\prime}_N(u)\|
\leq
%\|J^{'}(u)-\mathbb{J}^{'}_{N}(u)\|+\mathcal{O}(L^{-1/2})
\|J^{\prime}(u)-\mathbb{J}^{\prime}_{N}(u)\|+\mathcal{O}(\Delta t).
\end{aligned}
\end{equation*}
We have the following split for $\|J^{\prime}(u)-\mathbb{J}^{\prime}_N(u)\|^2$:
\begin{align*}
&\|J^{\prime}(u)-\mathbb{J}^{\prime}_N(u)\|^2 \nonumber\\
&=\int_{0}^{T}\left(\mathbb{E}\big[\hat{p}_tb^{\prime}_u(t)-\hat{P}_t^{K,L,\pi}\sum\limits_{n=0}^{N-1}b^{\prime}_u(t_n)\chi_{I_{n}^{N}}(t)\big]
+\mathbb{E}\big[\hat{q}_t\sigma^{\prime}_u\left(y_t,u(t)\right)-\hat{Q}_t^{\tilde{K},L,\pi}\sigma^{\prime}_u
\left(y_{t}^{\pi},u(t)\right)\big]\right)^2dt \nonumber\\
&\leq C\int_{0}^{T}\left(\left(\mathbb{E}\big[\hat{p}_tb^{\prime}_u(t)-
\hat{p}_t^{\pi}\sum\limits_{n=0}^{N-1}b^{\prime}_u(t_n)\chi_{I_{n}^{N}}(t)\big]\right)^2 
+\left(\mathbb{E}\big[ \big(\hat{p}_t^{\pi}-\hat{P}_t^{K,\pi}\big)\sum\limits_{n=0}^{N-1}b^{\prime}_u(t_n)\chi_{I_{n}^{N}}(t)\big]\right)^2  \right.\nonumber\\
&\quad\left.+\left(\mathbb{E}\big[ \big(\hat{P}_t^{K,\pi}-\hat{P}_t^{K,L,\pi}\big)\sum\limits_{n=0}^{N-1}b^{\prime}_u(t_n)\chi_{I_{n}^{N}}(t)\big]\right)^2\right)dt \nonumber\\
&\quad +C\int_{0}^{T}\left(\left(\mathbb{E}\big[\hat{q}_t\sigma^{\prime}_u\left(y_t,u(t)\right)-\hat{q}_t^{\pi}
\sigma^{\prime}_u\left(y_{t}^{\pi},u(t)\right)\big]\right)^2
+\left(\mathbb{E}\big[\big(\hat{q}_t^{\pi}-\hat{Q}_t^{\tilde{K},\pi}\big)\sigma^{\prime}_u\left(y_t^{\pi},u(t)\right)\big]\right)^2 \right. \nonumber\\
&\quad\quad\left.+\left(\mathbb{E}\big[\big(\hat{Q}_t^{\tilde{K},\pi}-\hat{Q}_t^{\tilde{K},L,\pi}\big)
\sigma^{\prime}_u\left(y_t^{\pi},u(t)\right)\big]\right)^2\right)dt \nonumber\\
&=C\int_0^T\sum\limits_{j=1}^6T_{j}dt.
\end{align*}
From the estimate (\ref{weak_pq}), we arrive at
\begin{align}
\begin{aligned}\label{first}
\int_0^TT_1dt
&=\sum\limits_{n=0}^{N-1}\int_{t_n}^{t_{n+1}}\left(\mathbb{E}\big[\hat{p}_tb^{\prime}_u(t)-
\hat{p}_nb^{\prime}_u(t_n)\big]\right)^2dt\\
&\leq2\sum\limits_{n=0}^{N-1}\int_{t_n}^{t_{n+1}}\left(\mathbb{E}\big[\hat{p}_{t_n}b^{\prime}_u(t_n)-\hat{p}_tb^{\prime}_u(t)\big]\right)^2dt
+2\Delta t
\sum\limits_{n=0}^{N-1}\left(\mathbb{E}\big[(\hat{p}_{n}-\hat{p}_{t_n})b^{\prime}_u(t_n)\big]\right)^2\\
&\leq2\sum\limits_{n=0}^{N-1}\int_{t_n}^{t_{n+1}}\left(\mathbb{E}\big[\hat{p}_{t_n}b^{\prime}_u(t_n)-\hat{p}_tb^{\prime}_u(t)\big]\right)^2dt+\mathcal{O}\big((\Delta t)^2\big).
%&\quad+2\Delta t\sum\limits_{n=0}^{N-1}\big( \mathbb{E}[\hat{\eta}_y^{'}(t_n,y_{t_n})b^{'}_u(t_n)(y_{n}-y_{t_n})+ \mathcal{O}(\Delta t)+\mathcal{O}(|y_n-y_{t_n}|^2)]\big)^2.\\
\end{aligned}
\end{align}
Utilizing (\ref{kac}), Taylor expansion and (\ref{estimate_sde}) yields
\begin{equation*}
\begin{aligned}
&\int_{t_n}^{t_{n+1}}\left(\mathbb{E}\big[\hat{p}_tb^{\prime}_u(t)-\hat{p}_{t_n}b^{\prime}_u(t_n)\big]\right)^2dt
=\int_{t_n}^{t_{n+1}}\left(\mathbb{E}\big[\hat{\eta}(t,y_t)b^{\prime}_u(t)-\hat{\eta}(t_n,y_{t_n})b^{\prime}_u(t_n) \big] \right)^2dt\\
%&=\int_{t_n}^{t_{n+1}}\Big(\mathbb{E}\big[      
%[\hat{\eta}b_{u}^{'}]_{t}^{'}(t_n,y_{t_n})(t-t_n)+[\hat{\eta}b_{u}^{'}]_{y}^{'}(t_n,y_{t_n})(y_t-y_{t_n})
%+\mathcal{O}\big((t-t_n)^2+(y_t-y_{t_n})^2\big)
%\big]\Big)^2dt\\
&=\int_{t_n}^{t_{n+1}}\left(\mathbb{E}\big[
[\hat{\eta}b_{u}^{\prime}]_{t}^{\prime}\left(t_n+\theta(t-t_n),y_{t_n}+\theta(y_t-y_{t_n})\right)(t-t_n) \right.\\
&\quad\left.+[\hat{\eta}b_{u}^{\prime}]_{y}^{\prime}\left(t_n+\theta(t-t_n),y_{t_n}+\theta(y_t-y_{t_n})\right)(y_t-y_{t_n})\big]\right)^2dt\\
&=\mathcal{O}\big((\Delta t)^3\big),
\end{aligned}
\end{equation*}
where $\theta\in(0,1)$. From the above estimate and (\ref{first}) yields
$\int_{0}^{T}T_1dt=\mathcal{O}\big((\Delta t)^2\big).$

According to (\ref{kac}), (\ref{estimate_sde1}), (\ref{weak_pq}) and Taylor expansion, we can bound $\int_{0}^{T}T_4dt$ as follows:
\begin{align*}
&\int_{0}^{T}T_4dt
=\sum_{n=0}^{N-1}\int_{t_n}^{t_{n+1}}\left(\mathbb{E}\big[\hat{q}_t\sigma_u^{\prime}\left(y_t,u(t_n)\right)-\hat{q}_n\sigma_{u}^{\prime}\left(y_{n},u(t_n)\right)\big]\right)^2dt\\
&\leq 2\sum_{n=0}^{N-1}\int_{t_n}^{t_{n+1}}\left(\mathbb{E}\big[\hat{q}_{t_n}\sigma_{u}^{\prime}\left(y_n,u(t_n)\right)-\hat{q}_t\sigma_u^{\prime}\left(y_t,u(t_n)\right) \big]\right)^2dt\\
&\quad+2\Delta t\sum_{n=0}^{N-1}\left(\mathbb{E}\big[(\hat{q}_n-\hat{q}_{t_n})\sigma_{u}^{\prime}\left(y_n,u(t_n)\right)  \big]\right)^2\\
&\leq2\sum_{n=0}^{N-1}\int_{t_n}^{t_{n+1}}\left(\mathbb{E}\left[
\left(\sigma\left(y_{t_n},u(t_n)\right)\hat{\eta}_{y}^{\prime}(t_n,y_{t_n})-\sigma\left(y_t,u(t_n)\right)\hat{\eta}_{y}^{\prime}(t,y_{t})\right)\sigma_{u}^{\prime}\left(y_t,u(t_n)\right)  \right.\right.\\
&\quad\left.\left.+\sigma\left(y_{t_n},u(t_n)\right)\hat{\eta}_{y}^{\prime}(t_n,y_{t_n})\sigma_{u,y}^{\prime\prime}\left(y_t+\theta(y_n-y_t),u(t_n)\right)(y_n-y_t)\right]\right)^2dt
+\mathcal{O}\big((\Delta t)^2\big)\\
&\leq C\sum_{n=0}^{N-1}\int_{t_n}^{t_{n+1}}\left(\mathbb{E}\big[
\sigma\left(y_{t_n},u(t_n)\right)\hat{\eta}_{y}^{\prime}(t_n,y_{t_n})-\sigma\left(y_t,u(t_n)\right)\hat{\eta}_{y}^{\prime}(t,y_{t})\big]\right)^2dt
+\mathcal{O}\big((\Delta t)^2\big).
\end{align*}
Analogously, employing (\ref{estimate_sde}) again yields 
\begin{align*}
&\sum_{n=0}^{N-1}\int_{t_n}^{t_{n+1}}\left(\mathbb{E}\big[
\sigma\left(y_{t_n},u(t_n)\right)\hat{\eta}_{y}^{\prime}(t_n,y_{t_n})-\sigma\left(y_t,u(t_n)\right)\hat{\eta}_{y}^{\prime}(t,y_{t})\big]\right)^2dt\\
&=\sum_{n=0}^{N-1}\int_{t_n}^{t_{n+1}}\left(\mathbb{E}\left[
\left(\hat{\eta}_{y}^{\prime}(t_n,y_{t_n})-\hat{\eta}_{y}^{\prime}(t,y_{t})\right)\sigma\left(y_t,u(t_n)\right) \right.\right.\\
&\quad\left.\left.+\hat{\eta}_{y}^{\prime}(t_n,y_{t_n})\sigma_{y}^{\prime}\left(y_t+\theta(y_{t_n}-y_t),u(t_n)\right)(y_{t_n}-y_t)
\right]\right)^2dt\\
&\leq \mathcal{O}\big((\Delta t)^2\big)+C\sum_{n=0}^{N-1}\int_{t_n}^{t_{n+1}}\left(\mathbb{E}\big[
\hat{\eta}_{y}^{\prime}(t_n,y_{t_n})-\hat{\eta}_{y}^{\prime}(t,y_{t})
\big]\right)^2dt\\
&=\mathcal{O}\big((\Delta t)^2\big)+C\sum_{n=0}^{N-1}\int_{t_n}^{t_{n+1}}\left(\mathbb{E}\big[
\hat{\eta}_{y,t}^{\prime\prime}\left(t+\theta(t_n-t),y_t+\theta(y_{t_n}-y_t)\right)(t_n-t) \right.\\
&\quad\left.+\hat{\eta}_{y,y}^{\prime\prime}\left(t+\theta(t_n-t),y_t+\theta(y_{t_n}-y_t)\right)(y_{t_n}-y_t)
\big]\right)^2dt\\
&=\mathcal{O}\big((\Delta t)^2\big),
\end{align*}
which implies that $\int_{0}^{T}T_4dt=\mathcal{O}\big((\Delta t)^2\big)$.

From Lemma \ref{error_pro} and Jensen inequality, the other terms are bounded as
\begin{align*}
\int_{0}^{T}(T_2+T_5)dt
&\leq C\sum\limits_{n=0}^{N-1}\Delta t\left(\mathbb{E}\big[\big(\hat{p}_{n}-\hat{P}_{n}^{K}(y_n) \big)^2\big] +\mathbb{E}\big[\big(\hat{q}_{n}-\hat{Q}_{n}^{\tilde{K}}(y_n) \big)^2\big]\right)\\
&\leq C\left(\sum\limits_{n=0}^{N-1}\mathbb{E}\big[\big(\hat{\mathcal{P}}_{0,n}\hat{p}_{n}-\hat{p}_{n}\big)^2\big]+\sum\limits_{n=0}^{N-1}\Delta t\mathbb{E}\big[\big(\hat{\mathcal{P}}_{1,n}\hat{q}_{n}-\hat{q}_{n}\big)^2\big]\right)
\end{align*}
and
\begin{align*}
\int_{0}^{T}(T_3+T_6)dt
&\leq C\sum\limits_{n=0}^{N-1}\Delta t\left(\mathbb{E}\big[\big(\hat{P}_{n}^{K}(y_n)-\hat{P}_{n}^{K,L}(y_n) \big)^2\big] +\mathbb{E}\big[\big(\hat{Q}_{n}^{\tilde{K}}(y_n)-\hat{Q}_{n}^{\tilde{K},L}(y_n) \big)^2\big]\right)\\
&\leq C\left(\max_{0\leq n\leq N}\mathbb{E}\big[\big(\hat{P}_{n}^{K}(y_n)-\hat{P}_{n}^{K,L}(y_n)\big)^2\big]+\sum\limits_{n=0}^{N-1}\Delta t
\mathbb{E}\big[\big(\hat{Q}_{n}^{\tilde{K}}(y_n)-\hat{Q}_{n}^{\tilde{K},L}(y_n)\big)^2\big]\right).
\end{align*}
Combining all the estimates above leads to result.
\end{proof}

\begin{Remark}\label{con_jieshi}
The analysis of the last error of Theorem \ref{JN}, i.e., the one induced by the simulation step of LSMC method, is rather involved and technical. \cite{goblt2,goblt3} mainly analyze this error and give a upper bound for generalized backward stochastic differential equation in Theorem 2. Although it is observed through numerical experiments that this upper bound may not be optimal, it is shown that we can achieve the desired order of convergence by setting a sufficiently large number of simulation paths and basis functions. Since the study on the optimal parameters is not the focus of this paper, we can refer to \cite{goblt2,goblt3} for more details. 

In summary, by setting enough simulation path $L$ (at least $\mathcal{O}(N^2)$) as well as basis function, the estimate $\varepsilon_N$ can reach first order convergence, which implies the first order convergence of control and multiplier can be obtained by Theorem \ref{u_uh} and \ref{mu_muh}.
\end{Remark}

\section{Numerical experiments}
Compared to the grid method, LSMC can handle SOCPs with higher dimension, such as some numerical examples with up to 10-dimensional state space are solved in \cite{higher}. Meanwhile, our theoretical framework is fully applicable to the grid method. From Theorem 4.6 in \cite{max3}, it holds that 
 $$\varepsilon_N=\sup_i\|J^{\prime}(u^{i,N})-J_{N}^{\prime}(u^{i,N})\|=\mathcal{O}(\Delta t)$$
with certain regularity assumptions, which implies that the control and multiplier can also reach first order convergence when the conditional expectations are approximated by the grid method.

In the following example, we use Algorithm \ref{lsmc_algorithm_xian} and grid method to solve several SOCPs including deterministic and feedback control problems, to verify our theoretical analysis. For more computational details on the grid method, see \cite{max3,grid2,grid3}.

\begin{Example}\label{exm1}
The following $d$-dimensional  deterministic stochastic optimal control problem is considered:
\begin{eqnarray*}
\begin{aligned}
\min\limits_{(y_t,u(t))\in K\times L^2([0,T];\mathbb{R}^d)} J\left(y_t,u(t)\right)=\frac{1}{2}\int_{0}^{T}\mathbb{E}\left[\sum\limits_{n=1}^{d}\left(y_t^n-\frac{1}{n}y_d\right)^2\right]dt
+\frac{1}{2}\int_{0}^{T}\sum\limits_{n=1}^{d}\left(u^{n}(t)\right)^2dt
\end{aligned}
\end{eqnarray*}
subject to
\begin{equation*}
\left\{\begin{aligned}
dy_t&=u(t)dt+\boldsymbol{\alpha} dW_t,\ t\in (0,T],\\ 
y_0&=0.
\end{aligned}\right.
\end{equation*}
\end{Example}
Here $y_t=[y_t^1,...,y_t^d]^{T}$ and $u(t)=[u^{1}(t),...,u^{d}(t)]^{T}$ are $d$-dimensional state variable and control variable, respectively. $\boldsymbol{\alpha}$ and $W_t$ are the $d\times d$ diagonal matrix whose diagonal element is constant $\alpha$ and $d$-dimensional standard Brownian motion. %$\delta=[\delta^{1},...,\delta^{d}]\in\mathbb{R}^{d}$ is the constraint parameter and the state constraint set $K$ is given as follows:
%$$K=\left\{y_t\in L_{\mathcal{F}}^2([0,T]\times\Omega;\mathbb{R}^d){|}\mathbb{E}\left[\int_0^Ty_t^jdt\right]\leq \delta^j,j=1,...,d\right\}.$$
We design this problem to have exact solution $u^*(t)=\left[T^2-t^2,\frac{T^2-t^2}{2},...,\frac{T^2-t^2}{d}\right]^{T}.$
To this end, the deterministic function $y_d$, state constraint parameter $\delta$ and multiplier $\mu^*$ are given by
\begin{equation*}
\begin{aligned}
&y_d=-\frac{1}{3}t^3+(T^2+2)t+\mu,\
\delta=\int_{0}^{T}\mathbb{E}\left[y_t^*\right]dt=\left[\frac{5}{12},\frac{5}{24},...,\frac{5}{12d}\right]^{T},\
\mu^{*}=\left[\mu,\frac{\mu}{2},...,\frac{\mu}{d}\right]^{T},
\end{aligned}
\end{equation*}
where $\mu$ is a chosen non-negative constant. 

We set $d=5$ and randomize $\mu=0.3$. Algorithm \ref{lsmc_algorithm_xian} is used to solved this problem and the basis functions of LSMC method are chosen as Voronoi partition basis function (\cite{goblt1}). The other parameters are set to $T=1,\alpha=0.1,\rho=0.5,\varepsilon_0=5*10^{-4},L=10^4$, $N=8,12,16,20,30,40$.
The exact solution and numerical solution for each dimension are shown in Figure \ref{figure3}. We can observe that our algorithm accurately captures the exact control in each dimension. The constraint for the state in each dimension is also shown in Figure \ref{figure3}, from which it is clear that the numerical states are all within the constraints. Figure \ref{figure3} also gives the convergence results for control and multiplier, which shows that the error decay in each dimension can reach the first order. Numerical results show that our algorithm is accurate in $5$-dimensional stochastic control problem with state constraint.

\begin{figure}[!htbp]
%\begin{minipage}[t]{0.45\linewidth}
%\centering
\flushleft
\label{3a}
\includegraphics[width=4cm,height=3.5cm]{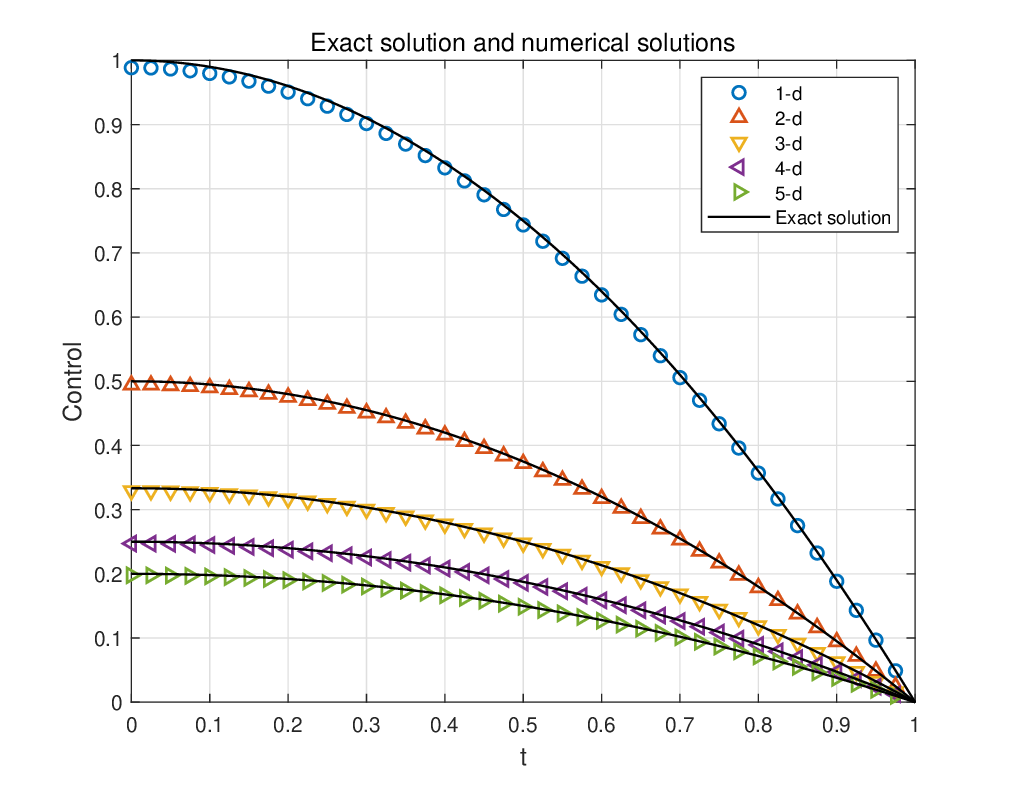}
\hspace{-5mm}
\label{3b}
\includegraphics[width=4cm,height=3.5cm]{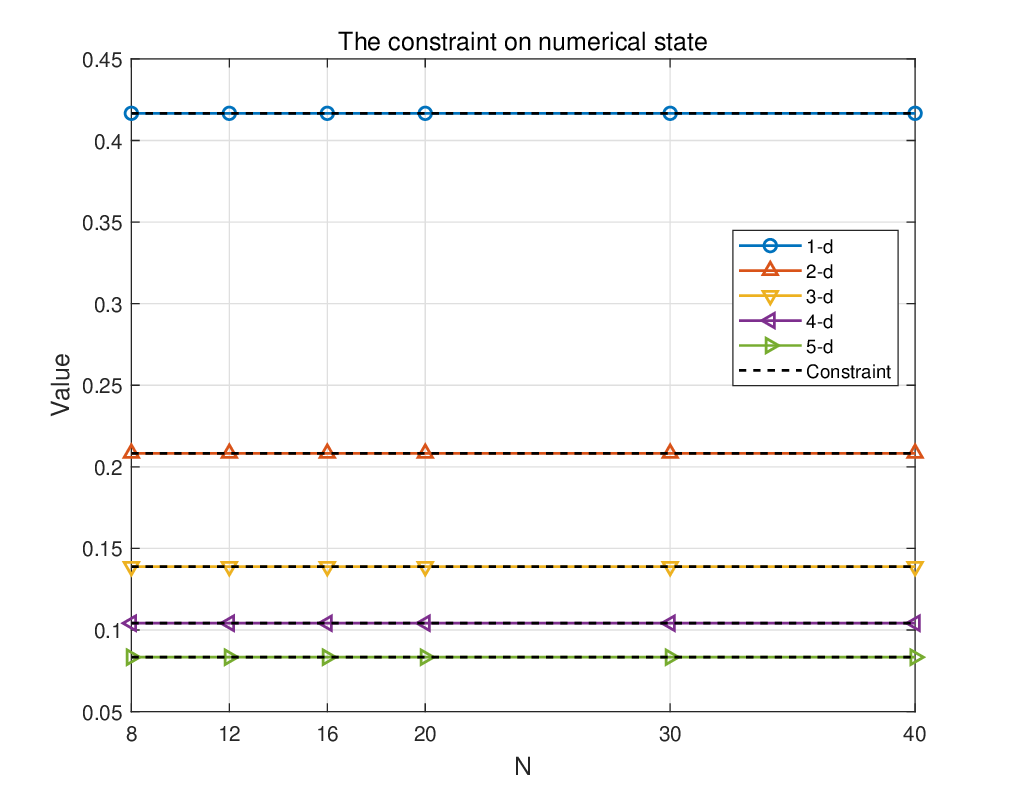}
\hspace{-5mm}
\label{3c}
\includegraphics[width=4cm,height=3.5cm]{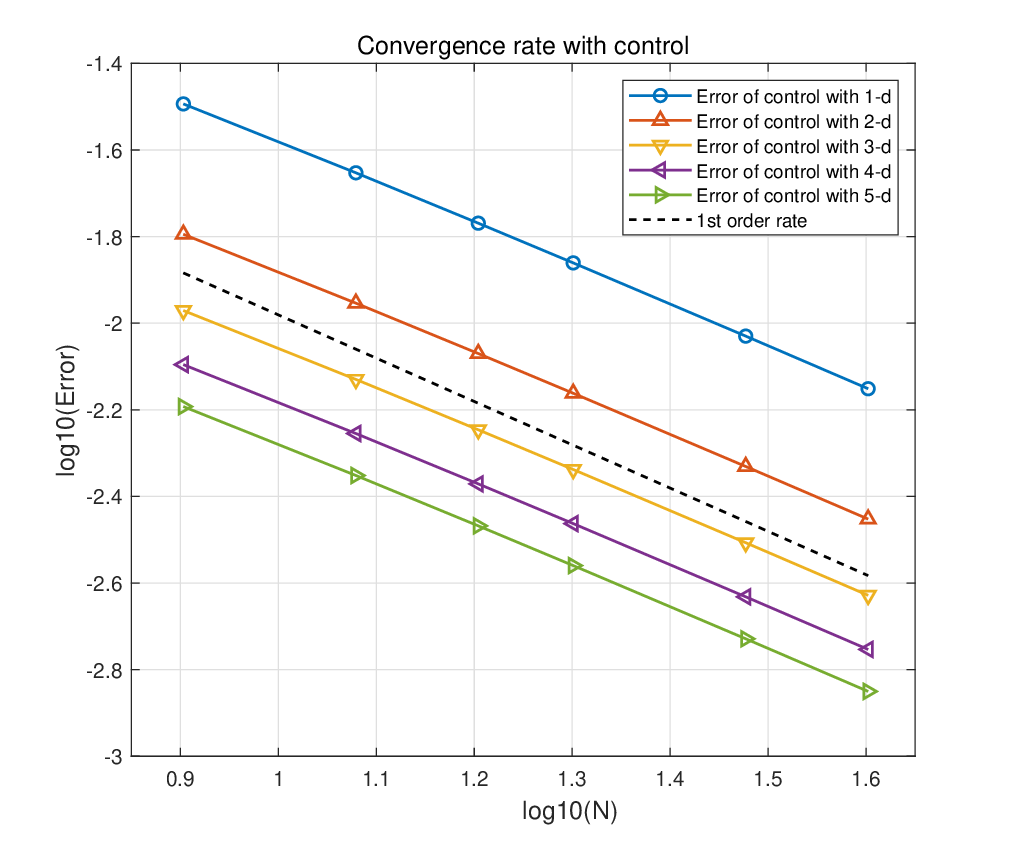}
\hspace{-5mm}
\label{3d}
\includegraphics[width=4cm,height=3.5cm]{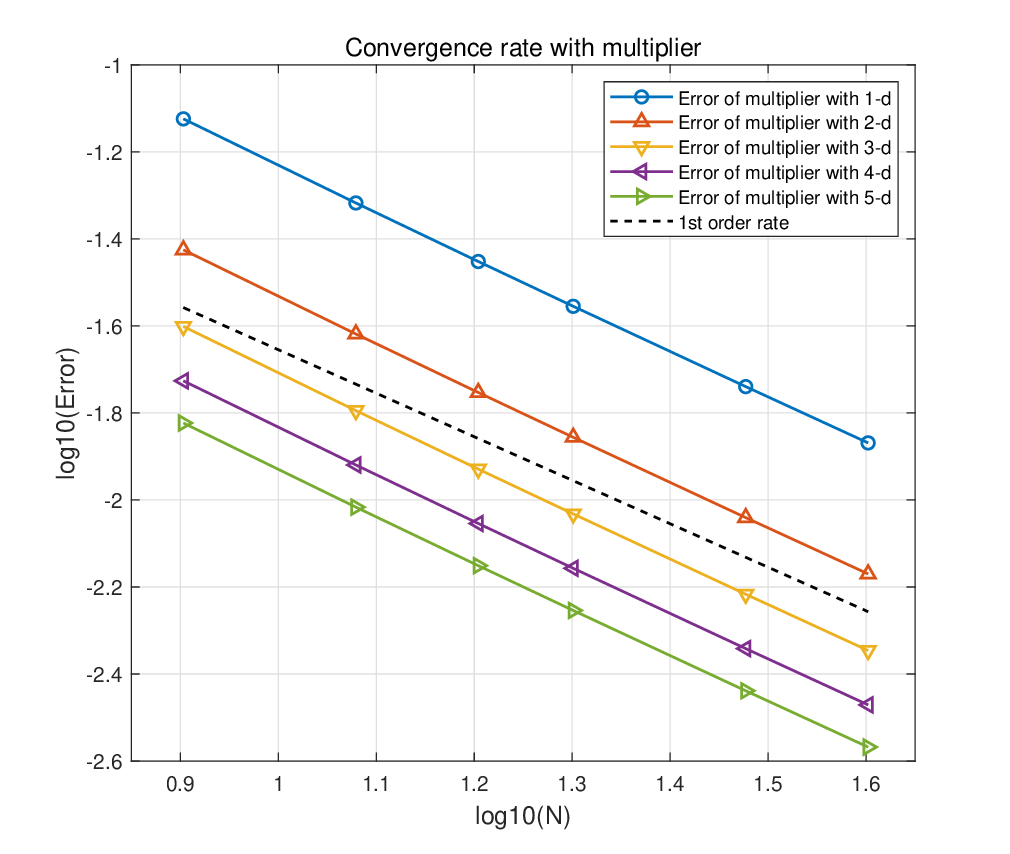}
\hspace{-5mm}
\caption{Left: The exact and numerical solutions with $N=40$. Center-left: Constraint testing of the numerical state. Center-right: The convergence rates of the control. Right: The convergence rates of multiplier for Example \ref{exm1}.}
\label{figure3}
\end{figure}

\begin{Example}(\cite{max3})\label{exm2}
The second stochastic optimal control problem with objective functional given in Example \ref{exm1} and $m=n=d=1$ is considered. The state equation is as follows:
\begin{equation*}
\left\{\begin{aligned}
dy_t&=\left(u(t)-r(t)\right)dt+\alpha u(t)dW_t,\ t\in (0,T],\\
y_0&=0.
\end{aligned}\right.
\end{equation*}
\end{Example}

We set $T=1,\alpha=0.1$ and the deterministic function $y_d,r(t)$ as well as optimal control $u^{*}$ are given as follows:
\begin{equation*}
\begin{aligned}
y_d=\frac{t}{2\alpha^2}-\frac{1}{2\alpha^4}\ln\frac{\alpha^2T+1}{\alpha^2(T-t)+1}+1+\mu^*,\quad u^*=\frac{T-t}{\alpha^2(T-t)+1},\quad r(t)=\frac{u^{*}}{2},
\end{aligned}
\end{equation*}
where $\mu^*=0.2$ is the chosen multiplier, and $\delta=\mathbb{E}[\int_0^Ty_t^*dt]\approx 0.16543$.

Algorithms \ref{lsmc_algorithm_xian} and grid method are used to solve this problem.
In the grid type method, we set $\rho=0.1$ and the tolerance error as well as numbers of Monte Carlo path are set $\varepsilon_0=5*10^{-5}$ and $L=10^{5}$. The number $N$ of time steps is taken as $N=30,40,...,80$, successively. The numerical results containing control error, multiplier error, convergence order, integral of the numerical state variable and computation time are reported in Table \ref{table1}.
In the LSMC method, two types of basis functions: Hypercubes (HC) and Voronoi partition (VP) (\cite{goblt1}) basis functions, are considered and we set $\rho=0.1$, $\varepsilon_0=10^{-4}$, $L=2000$, $N=8,12,16,20,30,40$. The numerical results are given in Table \ref{table2}.
It can be observed that the numerical states obtained by both methods are within the constraint set and the errors of the control and multiplier reach first-order convergence, which is consistent with our theoretical analysis. Meanwhile, it is clear that Algorithms \ref{lsmc_algorithm_xian} and grid method are computationally efficient while maintaining accuracy.

\begin{table}[!ht]
\footnotesize
\caption{Control error, multiplier error, convergence order, integral of the numerical state and computation time solved by grid method for Example \ref{exm2}.}
%\vspace{-5mm}
\begin{center}
\begin{tabular}
{|c| c| c| c| c| c| c| c|}
\hline\label{table1}

Method                &N         &Control error  &Rate   &Multiplier error   &Rate   &Integral     &Time        \\
\hline
\multirow{6}{*}{Grid method}   &$30$      &$7.04343\mathrm{e}-3$   &$\backslash$   &$1.23094\mathrm{e}-2$ &$\backslash$   &$0.16543$    &$3.26$s      \\
\cline{2-8}
                      &$40$      &$5.29748\mathrm{e}-3$  &$0.99$            &$9.24659\mathrm{e}-3$       &$0.99$       &$0.16543$    &$4.15$s      \\
\cline{2-8}
                       &$50$      &$4.24554\mathrm{e}-3$  &$0.99$            &$7.42457\mathrm{e}-3$       &$0.98$     &$0.16543$   &$5.26$s     \\
\cline{2-8}
                      &$60$      &$3.54257\mathrm{e}-3$  &$0.99$             &$6.21277\mathrm{e}-3$       &$0.98$      &$0.16543$   &$6.11$s        \\
\cline{2-8}
                      &$70$      &$3.03915\mathrm{e}-3$  &$0.99$             &$5.36264\mathrm{e}-3$       &$0.95$     &$0.16543$   &$7.50$s       \\
\cline{2-8}
                      &$80$      &$2.66152\mathrm{e}-3$  &$0.99$             &$4.67150\mathrm{e}-3$       &$1.03$     &$0.16543$  &$9.80$s    \\
\hline
\end{tabular}
\end{center}
\end{table}

\begin{table}[!ht]
\footnotesize
\caption{Control error, multiplier error, convergence order, integral of the numerical state and computation time solved by Algorithm \ref{lsmc_algorithm_xian} for Example \ref{exm2}.}
%\vspace{-5mm}
\begin{center}
\begin{tabular}
{|c| c| c| c| c| c| c| c|}
\hline\label{table2}

Basis functions            &N      &Control error  &Rate   &Multiplier error   &Rate  &Integral  &Time  \\
\hline
\multirow{6}{*}{HC}    &$8$      &$2.57203\mathrm{e}-2$   &$\backslash$   &$4.49237\mathrm{e}-2$        &$\backslash$   &$0.16543$   &$0.16$s         \\
\cline{2-8}
                      &$12$      &$1.73583\mathrm{e}-2$   &$0.97$            &$3.10487\mathrm{e}-2$       &$0.91$        &$0.16543$   &$0.54$s          \\
\cline{2-8}
                       &$16$      &$1.30960\mathrm{e}-2$  &$0.98$           &$2.38207\mathrm{e}-2$       &$0.92$         &$0.16543$   &$1.18$s      \\
\cline{2-8}
                      &$20$      &$1.05151\mathrm{e}-2$   &$0.98$            &$1.88635\mathrm{e}-2$       &$1.05$        &$0.16543$   &$2.26$s      \\
\cline{2-8}
                      &$30$      &$7.04469\mathrm{e}-3$   &$0.99$            &$1.27467\mathrm{e}-2$       &$0.97$        &$0.16543$   &$5.48$s        \\
\cline{2-8}
                      &$40$      &$5.30063\mathrm{e}-3$   &$0.99$            &$9.27519\mathrm{e}-3$       &$1.11$       &$0.16543$   &$15.39$s    \\
\hline
\multirow{6}{*}{VP}    &$8$      &$2.56788\mathrm{e}-2$   &$\backslash$    &$4.24494\mathrm{e}-2$        &$\backslash$    &$0.16543$   &$0.51$s         \\
\cline{2-8}
                      &$12$      &$1.73326\mathrm{e}-2$   &$0.97$            &$2.90962\mathrm{e}-2$       &$0.93$      &$0.16543$   &$1.35$s         \\
\cline{2-8}
                       &$16$      &$1.30837\mathrm{e}-2$  &$0.98$           &$2.24400\mathrm{e}-2$       &$0.90$        &$0.16543$   &$2.69$s         \\
\cline{2-8}
                      &$20$      &$1.05093\mathrm{e}-2$   &$0.98$            &$1.82058\mathrm{e}-2$       &$0.94$       &$0.16543$   &$4.53$s          \\
\cline{2-8}
                      &$30$      &$7.04458\mathrm{e}-3$   &$0.99$            &$1.20292\mathrm{e}-2$       &$1.02$       &$0.16543$   &$8.49$s        \\
\cline{2-8}
                      &$40$      &$5.29884\mathrm{e}-3$   &$0.99$            &$9.24822\mathrm{e}-3$       &$0.91$       &$0.16543$   &$18.02$s          \\
\hline
\end{tabular}
\end{center}
\end{table}

\begin{Example}\label{exm3}
The following stochastic optimal control problem with no explicit exact solution is considered:
\begin{eqnarray}\label{example3_object}
\begin{aligned}
\min\limits_{(y_t,u(t))\in K\times L^2([0,T];\mathbb{R})} J\left(y_t,u(t)\right)=\frac{1}{2}\int_{0}^{T}\mathbb{E}\left[\left(y_t-1-\mu^{*}\right)^2\right]dt+\frac{1}{2}\int_{0}^{T}\left(u(t)\right)^2dt
\end{aligned}
\end{eqnarray}
subject to
\begin{equation}\label{example3_state}
\left\{\begin{aligned}
dy_t&=\left(u(t)+y_t\right)dt+\alpha\sqrt{1+y_t^2}dW_t,\ t\in (0,T],\\ 
y_0&=1.
\end{aligned}\right.
\end{equation}
\end{Example}
We set $T=1$ and $\alpha=0.1$. Because an explicit solution cannot be written for this problem, we aim to construct a set of reference solution. The following unconstrained problem is first solved by using the grid type method with the parameter settings of $N=600$ and $L=5*10^{5}$:
\begin{eqnarray*}
\begin{aligned}
\min\limits J\left(y_t,u(t)\right)=\frac{1}{2}\int_{0}^{T}\mathbb{E}\left[\left(y_t-1\right)^2\right]dt+\frac{1}{2}\int_{0}^{T}
\left(u(t)\right)^2dt
\end{aligned}
\end{eqnarray*}
subject to state equation (\ref{example3_state}).
The solutions $\tilde{u}_h,\tilde{y}_h$ and $(\tilde{p}_h,\tilde{q}_h)$ for unconstrained problem are obtained. Then we set the constraint parameter $\delta=\mathbb{E}\left[\int_0^T\tilde{y}_h dt\right]\approx1.34150$, which implies that the solutions of unconstrained problem are also the solutions of SOCP (\ref{example3_object}) with the state constraint parameter $\delta=1.34150$. 

The multiplier $\mu^{*}$ is randomly chosen to be 1. In the grid type method, we set $\rho=0.1,\varepsilon_0=5*10^{-5},L=10^{5}$ and $N=20,25,30,40,50,60$. In the LSMC method, we set $\rho=0.1,\varepsilon_0=10^{-4},L=2000$ and $N=8,12,16,20,30,40$. 
The numerical results are given in Table \ref{table3} and \ref{table4}. The similar results are obtained as for the previous examples. 

To test the effect of the constraint parameter $\delta$, we set $\rho=1/i$, $\varepsilon_0=10^{-5}$, $\delta=1,0.5$ and keep the other parameters constant. The integral values of the numerical states solved by the two methods are given in Table \ref{table5}, from which it can be observed that the state variables obtained by both methods are in the constraint set, which verifies the effectiveness of our algorithm.

\begin{table}[!ht]
\caption{Control error, multiplier error, convergence order, integral of the numerical state and computation time solved by grid method for Example \ref{exm3}.}
%\vspace{-5mm}
\begin{center}
\begin{tabular}
{|c| c| c| c| c| c| c| c|}
\hline\label{table3}

Method         &N         &Control error  &Rate   &Multiplier error   &Rate  &Integral   &Time        \\
\hline
\multirow{6}{*}{Grid method} &$20$  &$2.65724e-2$   &$\backslash$ &$1.47905e-2$ &$\backslash$  &$1.34150$   &$4.17$s      \\
\cline{2-8}
                      &$25$      &$2.11674\mathrm{e}-2$  &$1.02$            &$1.20271\mathrm{e}-2$       &$0.93$       &$1.34150$   &$4.30$s      \\
\cline{2-8}
                       &$30$      &$1.75304\mathrm{e}-2$  &$1.03$            &$1.01448\mathrm{e}-2$       &$0.93$     &$1.34150$   &$4.89$s     \\
\cline{2-8}
                      &$40$      &$1.29447\mathrm{e}-2$  &$1.05$             &$7.40610\mathrm{e}-3$       &$1.09$      &$1.34150$   &$6.57$s        \\
\cline{2-8}
                      &$50$      &$1.01969\mathrm{e}-2$  &$1.07$             &$5.94427\mathrm{e}-3$       &$0.99$     &$1.34150$   &$8.33$s      \\
\cline{2-8}
                      &$60$      &$8.36104\mathrm{e}-3$  &$1.09$             &$4.94284\mathrm{e}-3$       &$1.01$     &$1.34150$   &$11.00$s    \\
\hline
\end{tabular}
\end{center}
\end{table}

\begin{table}[!ht]
\caption{Control error, multiplier error, convergence order, integral of the numerical state and computation time solved by Algorithm \ref{lsmc_algorithm_xian} for Example \ref{exm3}.}
%\vspace{-5mm}
\begin{center}
\begin{tabular}
{|c| c| c| c| c| c| c| c|}
\hline\label{table4}

Basis functions       &N      &Control error  &Rate   &Multiplier error   &Rate &Integral   &Time  \\
\hline
\multirow{6}{*}{HC}    &$8$      &$6.64573\mathrm{e}-2$   &$\backslash$   &$5.19389\mathrm{e}-2$        &$\backslash$   &$1.34150$   &$0.39$s         \\
\cline{2-8}
                      &$12$      &$4.44963\mathrm{e}-2$   &$0.99$           &$3.52938\mathrm{e}-2$       &$0.95$        &$1.34150$   &$1.46$s          \\
\cline{2-8}
                       &$16$      &$3.33019\mathrm{e}-2$  &$1.01$           &$2.65026\mathrm{e}-2$       &$1.00$         &$1.34150$   &$4.28$s      \\
\cline{2-8}
                      &$20$      &$2.65966\mathrm{e}-2$   &$1.01$            &$2.14768\mathrm{e}-2$       &$0.94$        &$1.34149$   &$8.94$s      \\
\cline{2-8}
                      &$30$      &$1.76264\mathrm{e}-2$   &$1.01$            &$1.43987\mathrm{e}-2$       &$0.99$        &$1.34150$   &$23.08$s        \\
\cline{2-8}
                      &$40$      &$1.29077\mathrm{e}-2$   &$1.08$            &$1.10728\mathrm{e}-2$       &$0.91$        &$1.34149$   &$64.81$s    \\
\hline
\multirow{6}{*}{VP}    &$8$      &$6.64066\mathrm{e}-2$   &$\backslash$   &$5.14559\mathrm{e}-2$        &$\backslash$   &$1.34150$   &$0.48$s         \\
\cline{2-8}
                      &$12$      &$4.46754\mathrm{e}-2$   &$0.98$           &$3.56011\mathrm{e}-2$       &$0.91$        &$1.34150$   &$1.47$s          \\
\cline{2-8}
                       &$16$      &$3.30933\mathrm{e}-2$  &$1.04$           &$2.75531\mathrm{e}-2$       &$0.89$         &$1.34150$   &$3.13$s      \\
\cline{2-8}
                      &$20$      &$2.66044\mathrm{e}-2$   &$0.98$            &$2.22703\mathrm{e}-2$       &$0.95$        &$1.34150$   &$5.51$s      \\
\cline{2-8}
                      &$30$      &$1.76770\mathrm{e}-2$   &$1.01$            &$1.49390\mathrm{e}-2$       &$0.98$        &$1.34150$   &$10.38$s        \\
\cline{2-8}
                      &$40$      &$1.33057\mathrm{e}-2$   &$0.99$            &$1.13490\mathrm{e}-2$       &$0.96$        &$1.34150$   &$20.69$s    \\
\hline
\end{tabular}
\end{center}
\end{table}

\begin{table}[!ht]
\caption{The integral values of numerical states solved by grid and LSMC methods with $\delta=1$ and $\delta=0.5$ for Example \ref{exm3}.}
%\vspace{-5mm}
\begin{center}
\begin{tabular}
{|c| c| c| c| c| c| c|}
\hline\label{table5}

N                          &20     &25  &30   &40   &50  &60        \\
\hline
{Grid method($\delta=1$)} &$1.00000$  &$1.00000$   &$1.00000$ &$1.00000$ &$1.00000$  &$1.00000$ \\ 
\hline
{Grid method($\delta=0.5$)}&$0.50000$  &$0.50000$   &$0.50000$ &$0.50000$ &$0.50000$  &$0.50000$ \\ 
\hline
{N}                          &8     &12  &16   &20   &30  &40        \\
\hline
{HC($\delta=1$)}&$1.00000$  &$1.00000$   &$1.00000$ &$1.00000$ &$1.00000$  &$1.00000$ \\ 
\hline
{HC($\delta=0.5$)}&$0.50000$  &$0.499999$   &$0.50000$ &$0.50000$ &$0.50000$  &$0.50000$ \\ 
\hline
{VP($\delta=1$)}&$1.00000$  &$1.00000$   &$1.00000$ &$1.00000$ &$1.00000$  &$1.00000$ \\ 
\hline
{VP($\delta=0.5$)}&$0.50000$  &$0.50000$   &$0.50000$ &$0.50000$ &$0.50000$  &$0.50000$ \\ 
\hline
\end{tabular}
\end{center}
\end{table}

%\begin{figure}[!htbp]
%%\begin{minipage}[t]{0.45\linewidth}
%%\centering
%\flushleft
%\label{4a}
%\includegraphics[width=4cm,height=3.5cm]{Ex3_solution.eps}
%\hspace{-5mm}
%\label{4b}
%\includegraphics[width=4cm,height=3.5cm]{Ex3_grid.eps}
%\hspace{-5mm}
%\label{4c}
%\includegraphics[width=4cm,height=3.5cm]{Ex3_hc.eps}
%\hspace{-5mm}
%\label{4d}
%\includegraphics[width=4cm,height=3.5cm]{Ex3_vp.eps}
%\hspace{-5mm}
%\caption{The numerical solutions with N=40 (Left) and the convergence rates of the control as well as multiplier solved by grid method (Center-left) and LSMC based on HC (Center-right) and VP (Right) for Example \ref{exm3}.}
%\label{figure4}
%\end{figure}

\begin{Example}(\cite{exm4})\label{exm4}
We consider the following linear quadratic feedback stochastic optimal control problem with expected  integral state constraint:
\begin{eqnarray*}
\begin{aligned}
\min\limits_{(y_t,u_t)\in K\times L_{\mathcal{F}}^2([0,T]\times\Omega;\mathbb{R})} J(y_t,u_t)=\frac{1}{2}\int_{0}^{T}\mathbb{E}\left[\left(y_t-\mu^{*}\right)^2\right]dt
\end{aligned}
\end{eqnarray*}
subject to
\begin{equation*}
\left\{\begin{aligned}
dy_t&=u_tdt+\alpha u_tdW_t,\ t\in (0,T],\\ 
y_0&=y^0,
\end{aligned}\right.
\end{equation*}
where $\mu^*$ is the multiplier.
\end{Example}

The optimal control, state and the corresponding optimal cost functional are given by
\begin{equation*}
\begin{aligned}
&u_t^*=-\frac{y_t^*}{\alpha^2},\
y_t^*=y^0e^{-\frac{t}{\alpha^2}-\frac{t}{2\alpha^2}-\frac{W_t}{\alpha}},\\
&J(y_t^*,u_t^*)=\frac{1}{2}\left((\mu^*)^2T+\left((y^0)^2\alpha^2-2\mu^* y^0\alpha^2\right)(1-e^{-\frac{T}{\alpha^2}})\right).
\end{aligned}
\end{equation*}
The grid type method is used to solved this problem. We set $T=1,y^0=1,\alpha=2,\mu^{*}=0.2,\rho=1/i$ and the state constraint parameter $\delta=\mathbb{E}\left[\int_0^Ty_t^*dt\right]=y^0\alpha^2(1-e^{-\frac{T}{\alpha^2}})\approx0.88480$. The tolerance error, Monte Carlo simulation path and grid node are set $\varepsilon_0=10^{-4},L=10^{5}$ and $N=10,20,30,40,50,60$, respectively. The errors of the multiplier and cost functional as well as the corresponding rates, computation time and integral of the numerical state are given in Table \ref{table6}. It is clearly shown that our algorithm is certainly effective and accurate for feedback stochastic optimal control and admits a first order rate of convergence.

\begin{table}[!ht]
\caption{Cost functional (CF) error, multiplier error, convergence order, integral of the numerical state and computation time solved by grid method for Example \ref{exm4}.}
%\vspace{-5mm}
\begin{center}
\begin{tabular}
{|c| c| c| c| c| c| c| c|}
\hline\label{table6}

Method                &N       &CF error  &Rate   &Multiplier error  &Rate  &Integral   &Time        \\
\hline
\multirow{6}{*}{Grid method} &$10$  &$1.12782\mathrm{e}-2$   &$\backslash$ &$6.19219\mathrm{e}-2$ &$\backslash$  &$0.88480$   &$4.32$s      \\
\cline{2-8}
                      &$20$      &$5.35812\mathrm{e}-3$  &$1.07$            &$3.50593\mathrm{e}-2$       &$0.82$      &$0.88480$   &$9.06$s      \\
\cline{2-8}
                       &$30$     &$3.46322\mathrm{e}-3$  &$1.08$            &$2.39449\mathrm{e}-2$       &$0.94$     &$0.88480$   &$12.01$s     \\
\cline{2-8}
                      &$40$      &$2.52560\mathrm{e}-3$  &$1.10$            &$1.81255\mathrm{e}-2$       &$0.97$     &$0.88480$   &$17.92$s       \\
\cline{2-8}
                      &$50$      &$1.98717\mathrm{e}-3$  &$1.07$            &$1.46049\mathrm{e}-2$       &$0.97$     &$0.88480$   &$26.37$s       \\
\cline{2-8}
                      &$60$      &$1.62600\mathrm{e}-3$  &$1.10$            &$1.22774\mathrm{e}-2$       &$0.95$     &$0.88480$   &$39.05$s    \\
\hline
\end{tabular}
\end{center}
\end{table}

\section{Conclusion}
In this work, a gradient projection method for solving stochastic optimal control problem with integral state constraint is proposed. First order optimality condition is derived by introducing Lagrange functional and a iterative scheme is given in which the specific multiplier is selected at each step to ensure that the state constraint holds. The fully discrete iterative scheme is obtained by piecewise constant approximation of the control variable and LSMC approximation of the conditional expectation in  FBSDEs. The first-order convergence of the control and multiplier can be achieved under certain regularity assumptions and the effectiveness of the algorithm and the accuracy of the theoretical findings are verified by numerical experiments. In our future work we will investigate the state-constrained optimal control problem governed by stochastic parabolic equation.

\section*{Data availibility}
The data that support the findings of this study are available from the authors upon reasonable
request.

\section*{Conflict of interest}
The authors declare that they have no conflict of interest.

\end{document}